\documentclass[11pt,usenames,dvipsnames]{amsart}
\usepackage{mathpazo}
\usepackage{hyperref}
\usepackage{cleveref}
\usepackage{bm}
\usepackage{comment}
\usepackage[margin=1in]{geometry}
\usepackage{enumitem}
\usepackage{calrsfs}
\DeclareMathAlphabet{\pazocal}{OMS}{zplm}{m}{n}

\usepackage{graphicx}
\usepackage{tikz-cd}

\usepackage[cmtip, all]{xy}
\usepackage{array}

\usepackage{amssymb,amsmath,latexsym,amsthm}

\DeclareMathOperator{\Mod}{Mod}

\DeclareMathOperator{\Aut}{Aut}
\DeclareMathOperator{\Jac}{Jac}

\DeclareMathOperator{\Span}{span}

\DeclareMathOperator{\id}{id}
\DeclareMathOperator{\image}{im}

\newcommand{\bk}[2]{\langle #1, #2 \rangle}

\newcommand{\pB}{\pazocal{B}}

\newcommand{\field}[1]{\mathbb{#1}}

\newcommand{\Z}{\field{Z}}

\newcommand{\R}{\field{R}}
\newcommand{\C}{\field{C}}

\newcommand{\cI}{\mathcal{I}}

\newcommand{\aaa}[3]{\alpha_{#1} \wedge \alpha_{#2} \wedge \alpha_{#3} }
\newcommand{\aab}[3]{\alpha_{#1} \wedge \alpha_{#2} \wedge \beta_{#3} }
\newcommand{\abb}[3]{\alpha_{#1} \wedge \beta_{#2} \wedge \beta_{#3} }
\newcommand{\bbb}[3]{\beta_{#1} \wedge \beta_{#2} \wedge \beta_{#3} }
\newcommand{\grF}[2]{\gr_{#1}^{#2}}

\newcommand{\bmu}{\boldsymbol \mu}

\newcommand{\LH}[1]{L/H}

\newcommand{\CCt}{\C(\!(t)\!)}
\newcommand{\cC}{\mathcal{C}}

\DeclareMathOperator{\gr}{gr}
\DeclareMathOperator{\coker}{coker}

\DeclareMathOperator{\bu}{\mathbf{u}}
\DeclareMathOperator{\bv}{\mathbf{v}}
\DeclareMathOperator{\bw}{\mathbf{w}}

\newtheorem{theorem}{Theorem}[section]
\newtheorem{lemma}[theorem]{Lemma}
\newtheorem{proposition}[theorem]{Proposition}
\newtheorem{corollary}[theorem]{Corollary}
\newtheorem*{theorem*}{Theorem}

\theoremstyle{definition}

\newtheorem*{question*}{Question}

\theoremstyle{remark}
\newtheorem{remark}[theorem]{Remark}

\newtheorem{example}[theorem]{Example}

\numberwithin{equation}{section}
\numberwithin{table}{section}
\numberwithin{figure}{section}

\newcommand\classification[2][]{%
  \gdef\@classification{%
    \href{http://www.ams.org/msc/}%
{\textit{2020 Mathematics Subject Classification}} \ignorespaces#2\unskip}}

\title{The Ceresa class and tropical curves of hyperelliptic type}

\author{Daniel Corey}
\email{corey@math.tu-berlin.de}
\address{TU-Berlin, Stra\ss e des 17 Juni 136, 10623 Berlin, Germany}

\author{Wanlin Li}
\email{liwanlin@crm.umontreal.ca}
\address{Wanlin Li: Centre de Recherches Math\'ematiques, 2920 chemin de la tour, Montreal, QC H3C 3J7, Canada}

\classification{14T25 (primary), and 14H30, 15C50, 57K20 (secondary)}

\begin{document}
	
\maketitle

\vspace{-1cm}

\begin{abstract}
    We define a new algebraic invariant of a graph $G$ called the Ceresa-Zharkov class and show that it is trivial if and only if $G$ is of hyperelliptic type, equivalently, $G$ does not have as a minor the complete graph on 4 vertices or the loop of 3 loops. After choosing edge-lengths, this class specializes to an algebraic invariant of a tropical curve with underlying graph $G$ that is closely related to the Ceresa cycle for an algebraic curve defined over $\CCt$.

    \medskip
    
    
    \medskip
    
    
\end{abstract}

\vspace{-.5cm}

\section{Introduction} 

Given a smooth genus $g\geq 2$ algebraic curve $C$ together with a point $p \in C$, there is a canonical null-homologous 1-cycle in its Jacobian $\Jac(C)$ obtained by taking the difference of the images of $C$ under the two Abel-Jacobi maps
\begin{equation*}
    C \hookrightarrow \Jac(C) \hspace{10pt} x\mapsto [x] - [p] \hspace{10pt} \text{and} \hspace{10pt} x\mapsto [p] - [x].
\end{equation*}
This is called the \textit{Ceresa cycle}, and we denote it by $C - C^{-}$. A landmark result of Ceresa in \cite{Ceresa} is that this cycle, for a very general curve of genus $g\geq 3$, is nontrivial in the Griffiths group of null-homologous cycles modulo algebraic equivalence. Nevertheless, it is always trivial for hyperelliptic curves, and it has long been conjectured that these are the only curves with trivial Ceresa cycle \cite[Question 8.5]{Hain85}, \cite[Remark 1.2]{BLLS}. Although several recent results \cite{Beauville21, BeauvilleSchoen, BLLS, LS} present nonhyperelliptic curves whose Ceresa cycles give rise to torsion classes in the  Griffiths group, this problem remains open. 
The goal of this paper is to probe the relationship between Ceresa triviality and hyperellipticity from the tropical viewpoint.

Degeneration techniques have long been used to study complex algebraic curves, both from the topological and  algebraic/arithmetic perspectives. Tropical geometry provides a systematic framework for recording the combinatorial data of a \textit{stable} degeneration as a \textit{tropical curve}, i.e., a graph with edge lengths and (possible) vertex weights. Topologically, stable degeneration is modeled by a family $\cC \to D$ of Riemann surfaces over a small complex disc that is holomorphic over $D\setminus x$ and the fiber over $x$ is a stable curve. Restricting to an infinitesimal neighborhood of $x$, we obtain a smooth algebraic curve $C$ over $\CCt$  that has stable reduction. The tropical curve corresponding to this degeneration is the (vertex-weighted) dual graph of the special fiber and its edge lengths record the speeds to which the nodes in the special fiber are formed.

In \cite{CEL}, the authors define a Ceresa class in the topological and tropical contexts that agrees with the $\ell$-adic Ceresa class---the image of the Ceresa cycle of a curve defined over $\CCt$ under the $\ell$-adic Abel-Jacobi map---after applying a suitable comparison morphism. In this paper, we define the  \textit{Ceresa-Zharkov} class $\bw(\Gamma)$ of a tropical curve $\Gamma$; this is a particular homomorphic image of the tropical Ceresa class, see \S \ref{sec:CeresaZharkov} for the precise formulation. Notably, nontriviality of the Ceresa-Zharkov class implies nontriviality of the tropical Ceresa class. 

In an effort to emulate the notion of generic real edge lengths, we define a graph-theoretic Ceresa-Zharkov class $\bw_{\tau}(G)$ (depending on a hyperelliptic quasi-involution $\tau$ of $\Sigma_g$, see \S\ref{sec:tropicalCeresaClass}), which lives in a module with coefficients in a polynomial ring whose variables correspond to the edges of $G$. Given a tropical curve $\Gamma$ with underlying graph $G$, (a representative of) the class $\bw(\Gamma)$ is obtained by evaluating $\bw_{\tau}(G)$ at the edge-lengths of $\Gamma$. We define what it means for $\bw_{\tau}(G)$ to be trivial---whence we call $G$ \textit{Ceresa-Zharkov} trivial---so that triviality of  $\bw_{\tau}(G)$ implies triviality of $\bw(\Gamma)$ for any tropical curve with underlying graph $G$. However, nontriviality of $\bw_{\tau}(G)$ does not imply nontriviality of $\bw(\Gamma)$ for any particular $\Gamma$ with underlying graph $G$, rather we view this as saying that $\bw(\Gamma)$ is \textit{generically} nontrivial. 

Our main theorem completely determines when a graph is Ceresa-Zharkov trivial, and it is closely related to hyperellipticity as we now explain. According to the tropical Torelli theorem, it is possible for nonisomorphic tropical curves to have isomorphic Jacobians as principally polarized tropical abelian varieties. A tropical curve whose Jacobian is isomorphic to that of a hyperelliptic tropical curve is said to be \textit{of hyperelliptic type}.   These tropical curves are defined in \cite{Corey}, where it is shown that being of hyperelliptic type depends only on the underling graph, is preserved when taking connected minors, and its forbidden minors are $K_4$ and $L_3$, see Figure \ref{fig:K4L3}. 

\begin{theorem*}[Theorem \ref{thm:WCTiffHET}]
A connected graph $G$ of genus $g\geq 2$ is Ceresa-Zharkov trivial if and only if $G$ is of hyperelliptic type, or equivalently, if and only if $G$ has no $K_4$ or $L_3$ minor.  
\end{theorem*}

Our choice for the name ``Ceresa-Zharkov'' class comes from a related construction due to Zharkov in \cite{zharkov}.  Independent of the tropical Ceresa class described above, Zharkov uses the tropical analogs of curves, Jacobians, the Abel-Jacobi map, and algebraic equivalence, to define and study a purely tropical Ceresa cycle.  The main result of his paper is that this ``tropical Ceresa cycle'' is not (tropically) algebraically equivalent to 0 for a generic tropical curve with a $K_4$ subgraph; here, $K_4$ is the complete graph on 4 vertices. As observed in \cite[Remark~7.3]{CEL}, when the underlying graph of $\Gamma$ is $K_4$, Zharkov's Ceresa cycle and notion of triviality coincides with our Ceresa-Zharkov class. This suggests the following question. 

\begin{question*}
What is the precise relationship between the Ceresa-Zharkov class and the tropical Ceresa cycle defined by Zharkov in \cite{zharkov}?
\end{question*}

Here is an outline of the paper. In \S \ref{sec:background}, we recall the construction of the topological and tropical Ceresa class from \cite{CEL}, then define the Ceresa-Zharkov class for tropical curves. Next, in \S \ref{sec:polynomialAlgebra}, we define and study the polynomial algebra in which the graph-theoretic Ceresa-Zharkov class is defined. We define the Ceresa-Zharkov class for graphs in \S \ref{sec:CeresaZharkov}, as well as study its elementary properties. Finally, we investigate the relationship between Ceresa-Zharkov triviality and being of hyperelliptic type in \S\ref{sec:MainTheorem}, where we prove the main theorem. 

\section*{Acknowledgments}
We thank Jordan Ellenberg for helpful discussions, especially at the beginning of this project. DC was supported by NSF-RTG grant 1502553 and the SFB-TRR project ``Symbolic Tools in
Mathematics and their Application'' (project-ID 286237555). WL was supported by the Centre de Recherches Math\'ematiques- Institut des Sciences Math\'ematiques postdoctoral fellowship. 

\section{Background}
\label{sec:background}

\subsection{Stable graphs, tropical curves and tropical Jacobians}\label{subsec:tropicalcurve}

Given a graph $G$, denote by $V(G)$ and $E(G)$ its set of vertices and edges, respectively. An edge of $G$ is a \textit{loop} if it is adjacent to a single vertex, and a pair of non-loop edges $(f,f')$ are \textit{parallel} if they join the same pair of vertices.  A (integral, unweighted) \textit{tropical curve} $\Gamma$ consists of a finite connected graph $G$, possibly with loops and parallel edges, together with a positive integer-valued function $c: E(G) \to \mathbb{Z}$ on the edge set $E(G)$. We view $G$ as the \textit{underlying graph} of $\Gamma$ and  $c(e)$ as the \textit{length} of the edge $e$. The \textit{genus} of $\Gamma$, written as $g(\Gamma)$, is the first Betti number of $G$, i.e., 
\begin{equation*}
    g(\Gamma) = |E(G)| - |V(G)| + 1.
\end{equation*}
A tropical curve $\Gamma$ is said to be \textit{2-connected} if $G$ has no cut-vertices; in particular, such a graph cannot have a loop or a bridge. 

Let $\Gamma = (G,c)$ be a genus $g\geq 2$ tropical curve and fix an orientation on $G$. 
The \textit{Jacobian} of $\Gamma$ is the real $g$-dimensional torus 
\begin{equation*}
    \Jac(\Gamma) = H^1(G,\R) / H^1(G,\Z)
\end{equation*}
together with the positive definite quadratic form $Q_{\Gamma}$ on $H_1(G,\R)$ given by
\begin{equation*}
    Q_{\Gamma}\left( \sum_{e\in E(G)} a_{e} \cdot e \right) = \sum_{e\in E(G)} a_e^2 \cdot c(e).
\end{equation*}

The valence of a vertex $v$ is the number of half edges adjacent to it; in particular a loop edge contributes 2 to the valence. A connected graph  is \textit{stable} if every vertex has valence at least $3$, and a tropical curve is stable if its underlying graph is stable. Two tropical curves are \textit{tropically equivalent} if one can be obtained from the other by a sequence of the following moves:
\begin{itemize}[noitemsep]
\item[-] adding or removing a 1-valent vertex and its incident edge, or 
\item[-] adding or removing a 2-valent vertex, preserving the underlying metric space. 
\end{itemize}
Every tropical curve of genus $g\geq 2$ is tropically equivalent to a unique stable tropical curve. 

A tropical curve is \textit{hyperelliptic} if there is an involution $\sigma$ of $\Gamma$ such that the quotient $\Gamma/\sigma$ (in the sense of \cite[\S~2.2]{Chan}) is a tree. A tropical curve is said to be \textit{of hyperelliptic type} if its Jacobian is isomorphic to the Jacobian of a hyperelliptic tropical curve as a principally polarized tropical abelian variety. This notion is defined and studied in \cite{Corey}, where it is shown that $\Gamma$ is of hyperelliptic type if and only if its underlying graph has no $K_4$ or $L_3$ minor (\cite[Theorem 1.1]{Corey}), see Figure \ref{fig:K4L3}. Moreover, if $\Gamma$ is a 2-connected tropical curve of hyperelliptic type, then there is a hyperelliptic tropical curve $\Gamma'$ that specializes to $\Gamma$ (\cite[Theorem 4.5]{Corey}). 

\begin{figure}
  \includegraphics[height=2.5cm]{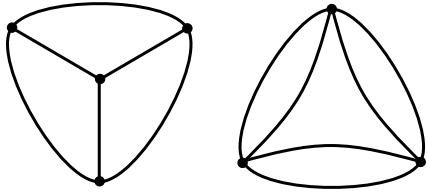}
  \caption{The graphs $K_4$ (left) and $L_3$ (right)}
  \label{fig:K4L3}
\end{figure}

\subsection{Dual graphs and multitwists from tropical curves}\label{subsec:multitwists}

We assume familiarity with surface topology and refer the reader to \cite{FarbMargalit} for a comprehensive treatment. Given an orientable topological real surface $S$, possibly with boundary or punctures, denote by $\Mod(S)$ its mapping class group. We write $\Sigma_g$ for a closed genus-$g$ surface and $\Sigma_g^1$ a genus $g$ surface with one boundary component. We always identify $\Sigma_g^1$ as the subsurface of $\Sigma_g$ obtained by removing a small open disc.  Given a curve $\gamma$ on $\Sigma_g$ or $\Sigma_g^1$, write $T_{\gamma}$ for the (left-handed) Dehn twist about $\gamma$.

Let  $\Lambda$ be a finite collection of pairwise disjoint isotopy classes of simple closed curves on $\Sigma_g$.  Its (unweighted) \textit{dual graph} is the graph with
\begin{itemize}
    \item[-] a vertex $v_S$ for every connected component $S$ of $\Sigma_g\setminus \bigcup_{\ell \in \Lambda} \ell$, and
    \item[-] an edge $e_{\ell}$ between $v_S$ and $v_{S'}$ for each $\ell$ in the boundary of $S$ and $S'$. The curve $\ell$ is said to be \textit{dual} to $e_\ell$.   
\end{itemize}
Any connected genus-$g$ graph is the dual graph of such a configuration, and $T_{\Gamma} = T_{\Gamma'}$ where $\Gamma'$ is the unique stable tropical curve tropically equivalent to $\Gamma$.
We are primarily interested in the case where $\Lambda$ is \textit{Lagrangian} in the sense of \cite{CEL}, i.e., each component $S$ of $\Sigma_g\setminus \bigcup_{\ell \in \Lambda} \ell$ has genus $0$. 

Let $\Lambda_1$ and $\Lambda_2$ be two Lagrangian arrangements of curves on $\Sigma_g$ as above such that their dual graphs, $G_1$ and $G_2$ respectively, are stable. Then $G_1$ is isomorphic to $G_2$ if and only if there is a mapping class of $\Sigma_g$ that takes $\Lambda_1$ to $\Lambda_2$; 
this follows from the well-known identification of the quotient of the curve complex of $\Sigma_g$ by $\Mod(\Sigma_g)$ with the tropical moduli space of genus $g$ tropical curves of total edge-length 1. 

Given a genus-$g$ tropical curve $\Gamma = (G,c)$ and $\Lambda$ an arrangement of curves on $\Sigma_g$ whose dual graph is $G$, define the multitwist
\begin{equation*}
    T_{\Gamma} = \prod_{e \in E(G)} T_{\ell_e}^{c(e)}.
\end{equation*}
By the previous paragraph and the fact that $\sigma T_{\ell}\sigma^{-1} = T_{\sigma(\ell)}$ for any $\sigma\in \Mod(\Sigma_g)$, the mapping class $T_{\Gamma}$ is well-defined up to conjugation in $\Mod(\Sigma_g)$.

\subsection{The Johnson homomorphism}
\label{sec:Johnson}

Let $\cI_g^1 \leq \Mod(\Sigma_g^1)$, resp. $\cI_g \leq \Mod(\Sigma_g)$, denote the Torelli group, and set
\begin{equation*}
    H = H_1(\Sigma_g^1,\Z) \cong H_1(\Sigma_g,\Z), \hspace{10pt} \text{ and } \hspace{10pt} L=\wedge^3H.
\end{equation*}
The intersection product on $H$ induces a $2$-form $\omega \in \wedge^2H$, and taking the exterior product with $\omega$ yields an injection $H\hookrightarrow L$.  The \textit{Johnson homomorphism} $J:\cI_g^1 \to L$, resp. $J:\cI_g \to L/H$, may be characterized in the following way. By \cite[Theorem 2]{Powell}, the Torelli group is generated by separating twists (Dehn twists about separating curves), and bounding pair maps (a product $T_{\gamma}T_{\gamma'}^{-1}$ where $\gamma,\gamma'$ form a separating pair). If $\gamma$ is a separating curve on $\Sigma_g$, then $J(T_{\gamma}) = 0$. Now suppose that $\gamma,\gamma'$ are a separating pair. The removal of $\gamma \cup \gamma'$ from $\Sigma_g^1$ separates $\Sigma_g^1$ into two surfaces $S$ and $S'$; let $S$ be the subsurface not containing the boundary component of $\Sigma_g^1$. The form $\omega$ restricts to the intersection 2-form on $S$, denote this by $\omega_S$. Then
\begin{equation*}
    J(T_{\gamma}T_{\gamma'}^{-1}) = \omega_{S} \wedge [\gamma]
\end{equation*}
where $[\gamma]$ is oriented so that $S$ appears on its right. The Johnson homomorphism on $\cI_g$ is defined in a similar way, except that one can choose either surface $S$ or $S'$ in the above formula; the two possible expressions are equivalent modulo $H$. 

\subsection{The tropical Ceresa class}
\label{sec:tropicalCeresaClass}

Next, we briefly recall the definition of the Ceresa class of $\Gamma$, \cite{CEL} for more details.  Fix a hyperelliptic quasi-involution $\tau$, i.e., a mapping class that acts as $-I$ on $H$. Consider the map 
\begin{equation*}
    \nu_{\tau}: \Mod(\Sigma_g) \to \LH{g} \hspace{20pt} \gamma \mapsto J([\gamma,\tau]).
\end{equation*}
This is a crossed homomorphism, and its class in $H^1(\Mod(\Sigma_g),\LH{g})$ is independent of the choice of $\tau$ \cite[Proposition~2.1]{CEL}. The \textit{$\tau$-Ceresa cocycle} of $\Gamma$, denoted by $\nu_{\tau}(\Gamma)$, is the restriction of $\nu_{\tau}$ to $\langle T_{\Gamma}\rangle \cong \Z$. Similarly, the \textit{Ceresa class} of $\Gamma$, denoted by $\nu(\Gamma)$, is the class of $\nu_{\tau}(\Gamma)$ in $H^1(\Z, \LH{g})$.

Let $Y$ be the Lagrangian subspace of $H$ spanned by the homology classes $\{[\ell] \, : \, \ell \in \Lambda\}$, and consider the following filtrations
\begin{equation*}
    F_q L = (\wedge^q Y) \wedge (\wedge^{3-q} H), \hspace{15pt} F_q (L/H) = \frac{F_{q}L + H}{H}, \hspace{15pt} \gr_q^{F}(L/H) = \frac{F_q(L/H)}{F_{q-1}(L/H)}.
\end{equation*}
Since $H\subset F_1L$ and $F_3L \cap H = \{0\}$, we have
{\footnotesize
\begin{equation}
\label{eq:explicitFq}
    F_0(L/H) = \frac{L}{H}, \hspace{10pt}     
    F_1(L/H) = \frac{F_1L}{H}, \hspace{10pt} 
    F_2(L/H) = \frac{F_2L+H}{H}, \hspace{10pt}
    F_3(L/H) \cong \frac{F_3L}{F_3L \cap H} \cong F_3L. \hspace{10pt}
\end{equation}
}
Then the map $\delta_{\Gamma}-I$ restricts to the graded components
\begin{equation*}
    \delta_{\Gamma}-I : \gr_{q}^F(L/H) \to \gr_{q-1}^F(L/H).
\end{equation*}
Define groups
\begin{align*}
    A(\delta_{\Gamma}) &= \image( H^1(\langle \delta_{\Gamma} \rangle, F_2(L/H)) \to H^1(\langle \delta_{\Gamma} \rangle, L/H)) \cong \frac{F_2L + H}{(\delta_{\Gamma}-I)(F_1L) + H} \\
    B(\delta_{\Gamma}) &= \coker(\delta_{\Gamma} - I: \gr_1^{F}(L/H) \to \gr_2^{F}(L/H)) \cong \frac{F_2L + H}{(\delta_{\Gamma}-I)(F_1L) + F_3L + H} \\
    C(\delta_{\Gamma}) &= \coker\left( (\delta_{\Gamma} - I)^2: \gr_1^{F}(L/H) \to \gr_3^{F}(L/H)\right) \cong \frac{F_3L}{(\delta_{\Gamma}-I)^2(F_1L)}
\end{align*}
and we have homomorphism $A(\delta_{\Gamma}) \to B(\delta_{\Gamma})$ \cite[Proposition~5.14]{CEL} and $(\delta_{\Gamma} - I) : B(\delta_{\Gamma}) \to C(\delta_{\Gamma})$. 
By \cite[Theorem~6.6]{CEL}, there is a hyperelliptic quasi-involution $\tau$ such that $\nu_{\tau}(\Gamma)$ lies in $F_2(L/H)$. In practice, whenever we compute the tropical Ceresa class, we always use such a $\tau$. In particular, the Ceresa class $\nu(\Gamma)$ lies in $A(\delta_{\Gamma}) \subset H^1(\langle \delta_\Gamma \rangle,L/H)$. Now we fix the following notation for future usage
\begin{equation*}
    \bv(\Gamma) = \text{image of } \nu(\Gamma) \text{ in } B(\delta_{\Gamma})  \hspace{20pt} \bw(\Gamma) = (\delta_{\Gamma}-I)(\bv(\Gamma)) \in C(\delta_\Gamma).
\end{equation*}
For a suitable $\tau$, the class $\bv(\Gamma)$ has a particularly nice representative
\begin{equation}
\label{eq:CeresaInBdelta}
    \bv_\tau(\Gamma) = \sum_{e\in E(G)} c(e)\, J([T_{\ell_e}, \tau]).
\end{equation}
The tropical curve $\Gamma$ is \textit{Ceresa trivial} if $v(\Gamma)=0$. If $\Gamma$ is Ceresa trivial then $\bv(\Gamma) = 0$. Given the above formula, $\bv(\Gamma)$ is easier to compute than $\nu(\Gamma)$, and therefore a good strategy to show that $\Gamma$ is Ceresa nontrivial is to show that $\bv(\Gamma) \neq 0$ in $B(\delta_{\Gamma})$.

\section{The action of a multitwist on some polynomial algebras} 
\label{sec:polynomialAlgebra}

\subsection{Multitwist on the homology of a surface}
\label{sec:multitwistHomology}
Throughout, given  $\Z$-modules $A$ and $R$, we set $A_R = A\otimes_{\Z} R$. Let $\Lambda$ be a collection of pairwise disjoint isotopy classes of simple closed curves on $\Sigma_{g}^{1}$ or $\Sigma_{g}$. Define a polynomial ring
\begin{equation*}
R[\Lambda] = \Z[x_{\ell} \, : \, \ell \in \Lambda]
\end{equation*}
which we denote by $R$ when $\Lambda$ is clear from the context.
For any loop $\ell\in\Lambda$, let $\delta_{\ell}: H_{R} \to H_{R}$ be the homomorphism defined on simple tensors by
\begin{equation}
\label{eq:deltaell}
    \delta_{\ell}(h\otimes a) = h\otimes a + \bk{h}{[\ell]} [\ell]  \otimes x_{\ell} a.  
\end{equation}
The map $\delta_{\ell}$ is an isomorphism with inverse
\begin{equation}
\label{eq:deltaellinv}
    \delta_{\ell}^{-1}(h\otimes a) = h\otimes a - \bk{h}{[\ell]} [\ell]  \otimes x_{\ell} a. 
\end{equation}
Because the loops in $\Lambda$ are pairwise disjoint, the $\delta_{\ell}$'s pairwise commute. The homomorphism $\delta_\ell$ is related to the symplectic representation of $T_{\ell}$ in the following way. Given a function $c:\Lambda \to \Z$, let $R\to \Z$ be the evaluation ring homomorphism $x_{\ell}\mapsto c(\ell)$. Then, under the identification $H \cong H\otimes_{\Z} R \otimes_{R} \Z$, the map $\delta_{\ell} \otimes_{R} \id_{\Z}$ equals the symplectic representation of $T_{\ell}^{c(\ell)}$.  

\begin{proposition}
\label{prop:prod2sum}
Given integers $\{a_{\ell} \, : \, \ell \in \Lambda \}$, we have
\begin{equation*}
    \left(\prod_{\ell \in \Lambda} \delta_{\ell}^{a_{\ell}} \right) - I =  \sum_{\ell\in \Lambda} a_{\ell} (\delta_{\ell}-I) \hspace{10pt} \text{ as maps } \hspace{10pt} H_R \to H_R.
\end{equation*} 
\end{proposition}

\begin{proof}
It suffices to show 
\begin{align*}
    &\delta_{\ell}\delta_{\ell'} - I = (\delta_{\ell} - I) + (\delta_{\ell'} - I),  \text{ and } \\ 
    &\delta_{\ell}^{-1} - I = -(\delta_{\ell}-I).
\end{align*}
The second formula follows readily from Formulas \eqref{eq:deltaell} and \eqref{eq:deltaell}. For the first formula, we have
\begin{align*}
    (\delta_{\ell}\delta_{\ell'}-I)(h\otimes a) &= \delta_{\ell}(h \otimes a + \langle h, [\ell'] \rangle [\ell'] \otimes x_{\ell'} a)-h\otimes a \\
    &=  (\delta_{\ell}-I)(h\otimes a) + \langle h, [\ell'] \rangle \delta_{\ell}( [\ell'] \otimes x_{\ell'} a) \\
    &= (\delta_{\ell}-I)(h\otimes a) + \langle h, [\ell'] \rangle ([\ell'] \otimes x_{\ell'} a + \langle [\ell'], [\ell] \rangle [\ell] \otimes x_{\ell}x_{\ell'} a) \\
    &= (\delta_{\ell}-I)(h\otimes a) + \langle h, [\ell'] \rangle[\ell']\otimes x_{\ell'} a \\
    &=(\delta_{\ell} - I)(h\otimes a) + (\delta_{\ell'} - I)(h\otimes a). \qedhere
\end{align*}
\end{proof}

Next, define $\pB=\langle \delta_{\ell} \, :\, \ell \in \Lambda \rangle$, which is a free abelian subgroup of $\Aut(H_R)$ whose rank is equal to the number of nonseparating loops in $\Lambda$. Recall that $Y=\Span_{\Z}\{[\ell] \, : \, \ell \in \Lambda\}$, which  is a Lagrangian subspace of $H$. The following is a straight-forward consequence of Equation \eqref{eq:deltaell} and Proposition \ref{prop:prod2sum}.

\begin{proposition}
\label{prop:fminusI}
For any $f\in \pB$, we have
\begin{equation*}
    \text{(1)} \;\; (f-I)(H_R) \subset Y_R, \hspace{20pt} \text{(2)} \;\; (f-I)(Y_R) = 0, \hspace{20pt} \text{(3)} \;\; (f-I)^2(H_R) = 0.
\end{equation*}
\end{proposition}

Denote by $\delta_{\Lambda}:H_R \to H_R$ the map
\begin{equation*}
    \delta_{\Lambda} = \prod_{\ell \in \Lambda} \delta_{\ell}.
\end{equation*}
By Proposition \ref{prop:prod2sum}, as a map $H_R\to H_R$, we have
\begin{equation*}
    \delta_{\Lambda}-I = \sum_{\ell \in \Lambda} (\delta_{\ell}-I).
\end{equation*}

\subsection{The action of $\delta_{\Lambda}$ on $L_R$}\label{subsec:action-on-L_R}

The third exterior power of $\delta_{\Lambda}$ is a $R$-module homomorphism $\wedge^3(H_R) \to \wedge^3(H_R)$. However, the natural $R$-module homomorphism 
\begin{equation*}
 \wedge^3(H_R) \to L_R \hspace{20pt} (h_1\otimes f_1) \wedge (h_2\otimes f_2) \wedge (h_3\otimes f_3) \mapsto (h_1\wedge h_2 \wedge h_3) \otimes (f_1f_2f_3). 
\end{equation*}
is an isomorphism; thus we may view the third exterior power of $\delta_{\Lambda}$ as an $R$-module endomorphism of $L_R$ which we will still denote by $\delta_\Lambda$. Next, we show that $\delta_{\Lambda}$ is an endomorphism of $L_R/H_R$, but we need the following more general setup. 

Let $\iota: S\hookrightarrow \Sigma_g$ be a subsurface (possibly with boundary) of $\Sigma_g$. The homology of $S$ splits as $H_{1}(S,\Z) \cong V \oplus W$ where  $V$ is the subgroup of $H$ generated by the boundary curves of $S$ and $W$ is a symplectic subspace of $H$ whose symplectic form, which we denote by $\omega_S \in \wedge^2W$, is obtained by restricting the intersection of $\Sigma_g$ to $S$. Let  $\pB_S$ be the subgroup of $\pB$ generated by the $\delta_{\ell}$ such that $\ell$ is isotopic to a curve in $S$. 

\begin{proposition}
\label{prop:BonOmega}
If $f\in \pB_S$, then 
\begin{enumerate}[noitemsep]
\item $f(\omega_S) \equiv \omega_S \mod (V \wedge W)_R$;
\item $(f-I)(h\wedge\omega_S) = (f-I)(h) \wedge \omega_S + f(h) \wedge \eta\ $ for some $\ \eta \in (V\wedge W)_R$.
\end{enumerate}
\end{proposition}
\noindent As a consequence, if $S=\Sigma_g$ and $f\in \pB$, then $(f-I)$, as an endomorphism $L_R\to L_R$, takes $H_R$ to $H_R$ (as an $R$-submodule of $L_R$). In particular, $\delta_\Lambda-I$ may be regarded as an $R$-module endomorphism of $L_R/H_R$. 

\begin{proof}
For (1), it suffices to show that 
\begin{equation}
\label{eq:deltaeSubsurface}
    \delta_{\ell}(\omega_{S}) \equiv \omega_S  \mod (V\wedge W)_{R}.
\end{equation}
Choose a symplectic basis $\alpha_1,\ldots,\alpha_g$, $\beta_1,\ldots,\beta_g$ of $H$ such that $\alpha_1,\ldots,\alpha_h,\beta_1,\ldots,\beta_h$ is a symplectic basis of $W$. This means that 
\begin{equation*}
    \omega_S = \sum_{i=1}^{h} \alpha_i \wedge \beta_i.
\end{equation*}
Write $[\ell] = \mu + \lambda$ where $\mu \in V$ and $\lambda \in W$. We have
\begin{equation*}
    \lambda = \sum_{i=1}^h \left( - \bk{\beta_i}{[\ell]} \alpha_i + \bk{\alpha_i}{[\ell]} \beta_i \right).
\end{equation*}
We compute
\begin{align*}
    \delta_{\ell}(\omega_S) &= \sum_{i=1}^{h} (\alpha_i + \bk{\alpha_i}{[\ell]} [\ell] \otimes x_{\ell} ) \wedge (\beta_i + \bk{\beta_i}{[\ell]} [\ell] \otimes x_{\ell} ) \\
    &= \omega_S + \sum_{i=1}^{h}\left(  \bk{\beta_i}{[\ell]}  \alpha_i\wedge [\ell] + \bk{\alpha_i}{[\ell]}  [\ell] \wedge \beta_i  \right) \otimes x_{\ell} \\
    &= \omega_S + [\ell] \wedge \sum_{i=1}^{h} \left(  -\bk{\beta_i}{[\ell]}  \alpha_i + \bk{\alpha_i}{[\ell]}   \beta_i  \right) \otimes x_{\ell} \\
    & = \omega_S + \mu \wedge \lambda \otimes x_{\ell},
\end{align*}
from which Formula \eqref{eq:deltaeSubsurface} follows. For (2),  we have 
\begin{align*}
    (f-I)(h\wedge \omega_{S})   = f(h) \wedge f(\omega_{S}) - h\wedge \omega_{S} = (f-I)(h) \wedge \omega_{S} + f(h) \wedge \eta. 
\end{align*}
for some $\eta \in (V\wedge W)_R$.
\end{proof}

Similar to the integral setup in Section \ref{sec:tropicalCeresaClass}, we define a filtration on $ L_R$ and $L_R/H_R$:
\begin{align*}
    F_q L_R := (F_qL)_R \cong (\wedge^qY_R)\wedge (\wedge^{3-q}H_R) \hspace{20pt} \text{ and } \hspace{20pt}  F_q (L_R/H_R) := (F_q(L/H))_R.
\end{align*}
and denote by $\grF{q}{F}(L_R/H_R)$ the graded piece
\begin{equation*}
    \grF{q}{F}(L_R/H_R) := \frac{F_{q}(L_R/H_R)}{F_{q+1}(L_R/H_R)}.
\end{equation*}

\begin{proposition}
\label{prop:fMinusIOnL}
Given $f\in \pB$ and $h_1,h_2,h_3\in H$, we have that
{\footnotesize
\begin{align*}
    (f-I)(h_1\wedge h_2 \wedge h_3) & \; = \;   (f - I)h_1 \wedge h_2 \wedge h_3 + h_1 \wedge (f-I) h_2 \wedge h_3 + h_1 \wedge h_2 \wedge (f - I) h_3  \\
    & +\;  (f - I)h_1 \wedge (f - I) h_2 \wedge h_3 + (f - I) h_1 \wedge  h_2 \wedge (f - I) h_3  + h_1 \wedge (f - I) h_2 \wedge (f - I) h_3 \\
    & +\; (f - I)h_1 \wedge (f - I) h_2 \wedge (f - I) h_3
\end{align*}
}
In particular, we have that $(f-I)(F_{q}(L_R)) \subset F_{q+1}(L_R)$.
\end{proposition}

As a consequence of this proposition, the map $\delta_{\Lambda} -I$ on $L_R/H_R$ takes $F_{q}(L_R/H_R)$ to $F_{q+1}(L_R/H_R)$, and hence induces a map on graded components:
\begin{equation*}
    \delta_{\Lambda}-I:\grF{q}{F}(L_R/H_R) \to \grF{q+1}{F}(L_R/H_R).
\end{equation*}

\begin{proposition}
As maps $\grF{q}{F}(L_R/H_R) \to \grF{q+1}{F}(L_R/H_R)$, we have
\begin{equation*}
    (\delta_{\Lambda}-I)(\alpha) = \sum_{\ell\in \Lambda} (\delta_{\ell}-I)(\alpha).
\end{equation*}
\end{proposition}

\begin{proof}
By Proposition \ref{prop:fMinusIOnL}, we have that 
{\footnotesize
\begin{equation*}
    (\delta_{\Lambda}-I)(h_1\wedge h_2 \wedge h_3) \equiv (\delta_{\Lambda}-I)h_1 \wedge h_2 \wedge h_3 + h_1 \wedge (\delta_{\Lambda}-I) h_2 \wedge h_3 + h_1 \wedge h_2 \wedge (\delta_{\Lambda}-I) h_3 \mod F_{q}(L_R/H_R). 
 \end{equation*}
 }
 Now apply Proposition \ref{prop:prod2sum}.
 \end{proof}

In summary, the maps $\delta_{\Lambda}-I$ and $\sum_{\ell \in \Lambda}(\delta_{\ell}-I)$ coincide when viewed as maps on $H_R$ or the graded components $\gr_{q}^F(L_R/H_R)$, but are different when viewed on the whole $L_R/H_R$.

\subsection{Configuration of loops from graphs}
\label{sec:graphs} 

Let $G$ be a connected genus-$g$ stable graph and $\Lambda = \{\ell_e \, : \, e\in E(G)\}$ a collection of pairwise disjoint isotopy classes of simple closed curves on $\Sigma_g$ whose dual graph is $G$, as in \S \ref{subsec:multitwists}. The polynomial ring 
\begin{equation*}
    R[G] = \Z[x_e \, : \, e\in E(G)]
\end{equation*}
is identified with $R[\Lambda]$ by the relabeling $x_{e} = x_{\ell_e}$. We simply write $R$ when the graph $G$ is clear from context.  
Define the endomorphism $\delta_{G}$ on $H_{R}$ (respectively, $L_R$ and $L_{R}/H_{R}$) by  $\delta_{G} = \delta_{\Lambda}$. The map $\delta_G$  is closely related to the polarization of the Jacobian of a tropical curve, as we see in the following construction. 

Enumerate the edges of $G$ by $e_1,\ldots,e_n$ so that $e_{g+1},\ldots,e_{n}$ are the edges of a spanning tree $T$. Fix an arbitrary orientation on $G$.  For $j=1,\ldots, g$, the graph $T\cup \{e_j\}$ contains a unique cycle $\gamma_j$; orient $\gamma_j$ in the direction of $e_j$. The cycles $\gamma_1,\ldots,\gamma_g$ form a basis of $H_1(G,\Z)$. Viewing $G$ as a 1-dimensional CW-complex, there is an embedding $\iota: G\hookrightarrow \Sigma_g$ so that $\iota(e_j) \cap \ell_j$ is a point. 

Orient $\ell_i$ so that the subsurface corresponding to the target of $e_i$ lies to the left of $\ell_i$, see Example \ref{ex:K4}. For $j=1,\ldots,g$, denote by $\alpha_j$ and $\beta_j$ the homology classes $-[\iota(\gamma_j)]$ and $[\ell_j]$, respectively; the orientations are chosen so that the signed intersection number of $\alpha_i$ with $\beta_i$ is 1. Then $\alpha_1,\ldots,\alpha_g,\beta_1,\ldots,
\beta_g$ is a symplectic basis for $H$.

With respect to this basis, we have
\begin{equation}
\label{eq:deltaQ}
    \delta_{G} =
    \begin{pmatrix} 
    I & 0 \\
    Q_{G} & I
    \end{pmatrix}.
\end{equation}
This allows us to view $Q_G$ as a map $Y^{\perp}_{R} \to Y_R$ where 
\begin{equation*}
    Y=\Span_{\Z}\{[\ell] \, : \, \ell \in \Lambda\} = \Span_{\Z}\{\beta_1,\ldots,\beta_g\}, \hspace{10pt} \text{ and } \hspace{10pt} Y^{\perp}=\Span_{\Z}\{\alpha_1,\ldots,\alpha_g\}
\end{equation*}
In particular, we have
\begin{equation*}
    Q_{G}(\alpha_j) = \sum_{i=1}^g q_{ij} \, \beta_i, \hspace{10pt}
    \delta_{G}(\alpha_j) = \alpha_j + Q_{G}(\alpha_j), \hspace{10pt} \text{ and } \hspace{10pt} \delta_{G}(\beta_j) = \beta_j.
\end{equation*}
Combining this with Proposition \ref{prop:fMinusIOnL}, we deduce the following identities:
{\small
\begin{equation}
\label{eq:deltaGQGformulas}
    \begin{array}{l}
      (\delta_G-I)(\aab{i}{j}{k}) = Q_G(\alpha_i) \wedge \alpha_j \wedge \beta_k + \alpha_i \wedge Q_G(\alpha_j) \wedge \beta_k +  Q_G(\alpha_i) \wedge Q_G(\alpha_j) \wedge \beta_k, \\
     (\delta_G-I)(\abb{i}{j}{k}) = Q_G(\alpha_i) \wedge \beta_j \wedge \beta_k, \\
     (\delta_G-I)^2(\aab{i}{j}{k}) = 2\, Q_G(\alpha_i) \wedge Q_G(\alpha_j) \wedge \beta_k.
    \end{array}
\end{equation}
}

\noindent The matrix $Q_{G}$ specializes to the polarization matrix $Q_{\Gamma}$ for a tropical curve $\Gamma=(G,c)$ in that, if $R\to \Z$ is the evaluation homomorphism $x_{e} \mapsto c(e)$, then $Q_{\Gamma} = Q_G \otimes_R \id_{\Z}$.  

Let us describe the entries of the matrix $Q_{G}$ with respect to this basis; for details and examples, 
see \cite[\S 4]{BologneseBrandtChua}. 
Let $x_i = x_{e_i}$. Each $q_{ij}$ (the $ij$-th entry of $Q_{G}$) is a linear form in the $x_k$'s: $x_k$ appears with coefficient $1$ if $\gamma_i$ and $\gamma_j$ traverse $e_k$ in the same direction, $-1$ if they traverse in the opposite direction, and $0$ if they do not meet at $e_k$. 

\begin{remark}
\label{rmk:QGLinear}
Denote by $R_d \subset R$ the subspace of homogeneous degree $d$-forms in $R$.  
Because the entries of $Q_{G}$ are linear forms in $R$, the map $\delta_G-I$ may be viewed as a map $H_{R_d} \to H_{R_{d+1}}$. 
\end{remark}

\begin{example}
\label{ex:K4}
Consider $G=K_4$, the complete graph on $4$ vertices, with the orientation as in Figure \ref{fig:k4toSurface}. To its right is an arrangement of curves on $\Sigma_3$ whose dual graph is $K_4$. Observe that each $\ell_i$ is oriented so that the subsurface corresponding to the target of $e_i$ lies to the left of $\ell_i$. Also observe that $\alpha_i$ is oriented so that the signed intersection product of $\alpha_i$ and $\beta_i:=[\ell_i]$ is 1. The matrix $Q_G$ is
\begin{equation}
\label{eq:QK4}
    Q_{G} = \begin{pmatrix}
    x_1 + x_5 + x_6 & -x_6 & -x_5 \\
    -x_6 & x_2+x_4+x_6 & -x_4 \\
    -x_5 & -x_4 & x_3+x_4+x_5
    \end{pmatrix}.
\end{equation}

\end{example}

\begin{figure}
    \centering
    \includegraphics[width=0.8\textwidth]{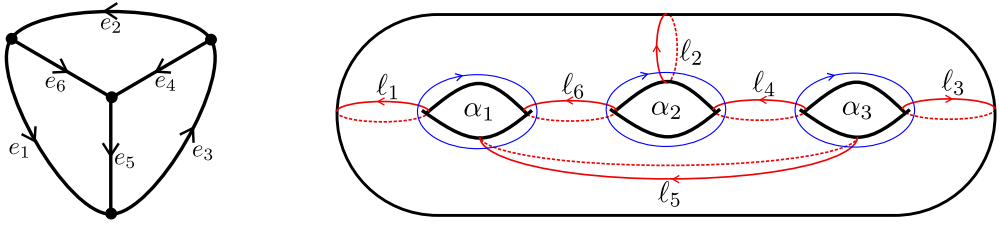}
    \caption{An arrangement of curves on $\Sigma_3$ with dual graph  $K_4$}
    \label{fig:k4toSurface}
\end{figure}

\section{The Ceresa-Zharkov class}
\label{sec:CeresaZharkov}

\subsection{Ceresa-Zharkov triviality}
\label{sec:CeresaZharkovTriviality}

Recall that we view $\Sigma_g^1$ as a subsurface of $\Sigma_g$ obtained by removing a small open disc $D$ from $\Sigma_g$.  
Recall that a hyperelliptic quasi-involution is an element of $\Mod(\Sigma_g^1)$ or $\Mod(\Sigma_g)$ whose action on $H = H_1(\Sigma_g^1,\Z) \cong H_{1}(\Sigma_g,\Z)$ is $-I$. 
Let $\tau'$ be a hyperelliptic quasi-involution on $\Sigma_g^{1}$ and $\tau$ its image under the natural map $\Mod(\Sigma_g^1) \to \Mod(\Sigma_g)$. 
Let $\Lambda'$ be a finite collection of pairwise disjoint isotopy classes of simple closed curves on $\Sigma_g^1$ and $\Lambda$ the image of this collection under the inclusion $\Sigma_g^1\hookrightarrow \Sigma_g$; note that any finite collection of simple closed curves can be isotoped away from the disc $D$, so $\Lambda$ may represent any finite collection of isotopy classes of simple closed curves on $\Sigma_g$. 

 The \textit{$\tau'$-Ceresa cocycle} of $\Lambda'$, respectively the \textit{$\tau$-Ceresa cocycle} of $\Lambda$, is defined as
\begin{equation*}
    \bmu_{\tau'}(\Lambda') = \sum_{\ell \in \Lambda'} J([T_{\ell},\tau']) \otimes x_{\ell} \in L_R, \hspace{5pt} \text{respectively,} \hspace{5pt} \bv_{\tau}(\Lambda) = \sum_{\ell \in \Lambda} J([T_{\ell},\tau]) \otimes x_{\ell} \in L_R/H_R 
\end{equation*} 
where $J$ is the Johnson homomorphism, see \S \ref{sec:Johnson}.
Define the \textit{$\tau$-Ceresa-Zharkov cocycle} as 
\begin{equation*}
    \bw_{\tau}(\Lambda) = (\delta_{\Lambda}-I)(\bv_{\tau}(\Lambda)).
\end{equation*}
When $G$ is a genus-$g$ graph, we write $\bv_{\tau}(G) = \bv_{\tau}(\Lambda)$ and $\bw_{\tau}(G) = \bw_{\tau}(\Lambda)$ where $\Lambda$ is an arrangement of curves whose dual graph is $G$. 
We say that $G$ is \textit{Ceresa-Zharkov trivial}  if there is a $\tau$ such that $\bw_{\tau}(G) = 0$. While our primary interest is in $\bw_{\tau}(G)$, we occasionally need the class $\bmu_{\tau'}(\Lambda')$ in \S \ref{sec:MainTheorem}.

In practice, when we compute the tropical Ceresa class for a tropical curve $\Gamma$, we use a hyperelliptic involution $\tau$ such that $\nu_{\tau}(\Gamma)$ lies in $F_2(L/H)$ as in \S \ref{sec:tropicalCeresaClass}. 
We do something similar in the graph-theoretic case. As before, $R_d\subset R$ denotes the linear subspace of degree $d$ forms. 

\begin{proposition}
\label{prop:CeresaF2R1}
There exists a hyperelliptic quasi-involution $\tau'$ of $\Sigma_g^1$ such that 
\begin{equation*}
    \bmu_{\tau'}(\Lambda') \in F_2L \otimes R_1
\end{equation*}
In particular, $\bv_{\tau}(G)$ lies in $F_2(L/H) \otimes R_1$ and $\bw_{\tau}(G)$ lies in $F_3L \otimes R_2$ where $\tau$ is the image of $\tau'$ under $\Mod(\Sigma_g^1) \to \Mod(\Sigma_g)$. 
\end{proposition}

\begin{proof}
By \cite[Theorem~6.6]{CEL}, there is a hyperelliptic quasi-involution $\tau'$ such that each $J([T_{\ell},\tau'])$ lies in $F_2L$. The proposition now follows from Remark \ref{rmk:QGLinear}.
\end{proof}

Here is a characterization for Ceresa-Zharkov triviality when we use such a $\tau$. 

\begin{proposition}
\label{prop:sufficientWCT}
If $G$ is Ceresa-Zharkov trivial and  $\bv_{\tau}(G) \in F_2(L/H) \otimes R_1$ then  
\begin{equation}\label{eq:wtau}
    \bw_{\tau}(G) \in  (\delta_{G}-I)^2(F_1(L/H)).
\end{equation}
Conversely, if there exists a $\tau$ such that equation \eqref{eq:wtau} holds, then $G$ is Ceresa-Zharkov trivial.
\end{proposition}

To prove this proposition, we first derive a characterization of  Ceresa-Zharkov triviality that works for \textit{any} $\tau$ but is more cumbersome to state due to the fact that the endomorphism  $\delta_G-I$ of $L_R/H_R$ does not split as the sum $\sum_{e\in E(G)}(\delta_e-I)$ as  discussed in \S \ref{subsec:action-on-L_R}. 

Denote by $\psi_G:L_R/H_R \to L_R/H_R$ the composition
\begin{equation*}
    \psi_G = (\delta_G-I) \circ \left(\sum_{e\in E(G) }(\delta_e-1) \right). 
\end{equation*}

\begin{lemma}
\label{lem:conditionsWCT}
The following are equivalent.
\begin{enumerate}
    \item The graph $G$ is Ceresa-Zharkov trivial;
    \item for every hyperelliptic quasi-involution $\tau$, we have  $\bw_{\tau}(G) \in \psi_G(L/H)$;
    \item there is a hyperelliptic quasi-involution $\tau$ such that $\bw_{\tau}(G) \in \psi_G(L/H)$.
\end{enumerate}
\end{lemma}

\begin{proof}
Given hyperelliptic quasi-involutions $\tau, \tilde{\tau}$, we have
\begin{equation*}
    (J([T_{\ell_e},\tilde{\tau}]) - J([T_{\ell_e},\tau])) \otimes x_e =  (T_{\ell_e}-I)_*(J(\tilde{\tau}^{-1} \tau)) \otimes x_e = (\delta_e-I)(J(\tilde{\tau}^{-1} \tau))
\end{equation*}
and therefore
\begin{equation}
\label{eq:tautautilde}
    \bv_{\tilde{\tau}}(G) - \bv_{\tau}(G) = \sum_{e\in E(G)} (\delta_{e}-I)(J(\tilde{\tau}^{-1} \tau)).
\end{equation}
Suppose $G$ is Ceresa-Zharkov trivial, say $\bw_{\tau}(G)= 0$ for some $\tau$. Applying $\delta_G-I$ to Formula \eqref{eq:tautautilde}, we see that (1) $ \Rightarrow$ (2). The implication (2) $\Rightarrow$ (3) is trivial. Now suppose (3) is true, say $\bw_{\tau}(G) =  \psi_G(\bu)$ for some $\bu \in L/H$. By surjectivity of the Johnson homomorphism, there is a $t\in \cI_g$ such that $J(t)=\bu$. The function $\tilde{\tau} = \tau t$ is a hyperelliptic quasi-involution, and by the above equation, we have
\begin{equation*}
    \bv_{\tilde{\tau}}(G)  = \bv_{\tau}(G) - \sum_{e\in E(G)} (\delta_{e}-I)(\bu)
\end{equation*}
Applying $\delta_G-I$, we get 
\begin{equation*}
\bw_{\tilde{\tau}}(G) = \bw_{\tau}(G)  - \psi_G(\bu) = 0,    
\end{equation*}
as required.  
\end{proof}

\begin{proof}[Proof of Proposition \ref{prop:sufficientWCT}]
Let $\alpha_1,\ldots,\alpha_g,\beta_1,\ldots,\beta_g$ be the symplectic basis of $H$ from \S \ref{sec:graphs}. 
We first record the image of the simple wedges under the map $\psi_G$ using Proposition \ref{prop:fMinusIOnL}:
\begin{align*}
    \psi_G(\aaa{i}{j}{k}) = &2(Q_{G} \alpha_i \wedge Q_{G} \alpha_j \wedge \alpha_k + Q_{G} \alpha_i \wedge  \alpha_j \wedge Q_{G} \alpha_k +  \alpha_i \wedge Q_{G} \alpha_j \wedge Q_{G}\alpha_k) \\
    &\hspace{30pt} +3 \, Q_G \alpha_i \wedge Q_{G} \alpha_j \wedge Q_{G} \alpha_k \\
    \psi_G(\aab{i}{j}{k}) = &2\, Q_{G} \alpha_i \wedge Q_{G} \alpha_j \wedge \beta_k \\
    \psi_G(\abb{i}{j}{k}) = &\psi_G(\bbb{i}{j}{k}) = 0.
\end{align*}
By this computation, Formula \eqref{eq:deltaGQGformulas}, and the fact that $(\delta_G-I)^2|_{F_2(L/H)} = 0$, we deduce that $  (\delta_{G}-I)^2|_{F_1(L/H)} = \psi_{G} |_{F_1(L/H)}$. Therefore,
\begin{equation}
\label{eq:psiF1}
 (\delta_{G}-I)^2(F_1(L/H)) = \psi_G(F_1(L/H)) = \psi_G(L/H) \cap (F_3(L/H) \otimes R_2).
\end{equation}
Suppose $G$ is Ceresa-Zharkov trivial and that $\bv_{\tau}(G) \in F_2(L/H) \otimes R_1$. Then the class $\bw_{\tau}(G)$ lies in $\psi_G(L/H) \cap (F_3(L/H) \otimes R_2)$ by Lemma \ref{lem:conditionsWCT}. So Formula \eqref{eq:wtau} follows from \eqref{eq:psiF1}, as required. Conversely, if there exists a $\tau$ such that Formula \eqref{eq:wtau} holds, then $G$ is Ceresa-Zharkov trivial by Lemma \ref{lem:conditionsWCT} and Formula \eqref{eq:psiF1}. 
\end{proof}

\subsection{Edge-contraction}
In this section, we prove that Ceresa-Zharkov triviality is preserved under edge contraction. We prove in \S \ref{sec:MainTheorem} that Ceresa-Zharkov triviality is a minor-closed property. Suppose $G$ is a connected graph, $f\in E(G)$ is a nonloop edge, and $G/f$ the graph obtained from $G$ by contracting $f$.  Denote by $\zeta_f:R[G] \to R[G/f]$ the map that evaluates $x_f$ to 0; this is a retract of the natural inclusion $R[G/f] \subset R[G]$.

\begin{proposition}
\label{prop:weakCeresaContract}
With $G$ and $f$ as above, 
we have an equality of maps $(L/H)_{R[G]} \to (L/H)_{R[G/f]} $:
\begin{equation}
\label{eq:contractf}
	(1\otimes \zeta_f)\circ(\delta_{G}-I) = (\delta_{G/f}-I)\circ(1\otimes \zeta_f), 
\end{equation}
Also,
\begin{equation*}
(1\otimes \zeta_f) (\bv_{\tau}(G)) = \bv_{\tau}(G/f). 
\end{equation*}
In particular, if $G$ is Ceresa-Zharkov trivial, then so is $G/f$. 
\end{proposition}

Note that $G$ and $G/f$ have the same genus $g$ and both can be embedded into $\Sigma_g$ as described in \S \ref{sec:graphs}. So the same $\tau \in \Mod(\Sigma_g)$ gives Ceresa cocycles of $G$ and $G/f$.

\begin{proof}
First, consider the maps $1\otimes \zeta_f$ and $\delta_e-I$ on $H_R$. We have
\begin{equation*}
    (1\otimes \zeta_f) (\delta_e-I) =  (\delta_e-I) (1\otimes \zeta_f) \hspace{10pt} \text { when } e\neq f, \hspace{10pt} \text{ and } (1\otimes \zeta_f) (\delta_f-I) = 0.
\end{equation*}
As a map $H_R \to H_R$, by Proposition \ref{prop:prod2sum}, we have
\begin{equation*}
    (1\otimes \zeta_f) (\delta_G-I) = (1\otimes \zeta_f) \sum_{e\in E(G)}(\delta_e-I) = \sum_{e\in E(G/f)} (\delta_e-I) (1\otimes \zeta_f) = (\delta_{G/f}-I)(1\otimes \zeta_f).
\end{equation*}
Thus the identity in Formula \eqref{eq:contractf} holds as maps $H_{R[G]} \to H_{R[G/f]}$, and one readily extends this as an identity of morphisms $(L/H)_{R[G]} \to (L/H)_{R[G\setminus f]}$ using Proposition \ref{prop:fMinusIOnL}.
Next, consider what happens for the Ceresa class. We have
\begin{align*}
    (1\otimes \zeta_f)(\bv_{\tau}(G)) = (1\otimes \zeta_f) \sum_{e\in E(G)} J([T_{\ell_e}, \tau]) \otimes x_e  = \sum_{e\in E(G/f)} J([T_{\ell_e}, \tau]) \otimes x_e = \bv_{\tau}(G/f).
\end{align*}
Finally, suppose $G$ is Ceresa-Zharkov trivial, say $\bw_{\tau}(G) = 0$ for some $\tau$.  Then,
	\begin{align*}
	(\delta_{G/f}-I)(\bv_{\tau}(G/f)) = (\delta_{G/f}-I)\circ (1\otimes \zeta_f)(\bv_{\tau}(G)) = (1\otimes \zeta_f) \circ (\delta_{G}-I) (\bv_{\tau}(G)) = 0,
	\end{align*}
	and therefore $G/f$ is Ceresa-Zharkov trivial. 
\end{proof}

\subsection{Tropical equivalence}
Similar to the setting of tropical curves, two graphs are \textit{tropically equivalent} if one can be obtained from the other via the following moves:
\begin{itemize}
    \item[-] adding or removing a 1-valent vertex together with its adjacent edge, 
    \item[-] subdividing a edge,
    \item[-] contracting exactly one edge adjacent to a 2-valent vertex. 
\end{itemize}

In this section, we show Ceresa-Zharkov triviality is preserved under all these moves and thus is a property of tropical equivalence classes. This fact is analogous to \cite[Lemma 4.4]{CEL} which asserts that tropically equivalent tropical curves have the same Ceresa class.

\begin{lemma}
\label{lem:contractSeparatingEdge}
    Let $G$ be a connected graph and $e$ a separating edge. Then 
    \begin{enumerate}
        \item $\delta_{G} = \delta_{G/e}$;
        \item $\bv_{\tau}(G) = \bv_{\tau}(G/e)$;
        \item $\bw_{\tau}(G) = \bw_{\tau}(G/e)$.
    \end{enumerate}
    In particular, the graph $G$ is Ceresa-Zharkov trivial if and only if its 2-edge connectivization $G^2$, which is obtained by contracting all separating edges of $G$, is Ceresa-Zharkov trivial.  
\end{lemma}

\begin{proof}
Statement (1) follows from the fact that  $\delta_{e} = I$ whenever $e$ is a separating edge. Statement (2) follows from the fact that $[T_{\ell_e}, \tau] = 1$ for any hyperelliptic quasi-involution $\tau$. Finally, statement (3) is a consequence of (1) and (2). 
\end{proof}

Let $G$ be a graph, and suppose $f \in E(G)$ is subdivided into 2 edges $e_1,e_2$, producing the graph $G'$.  Consider the ring map $\phi_f:R[G] \to R[G']$ given by 
\begin{equation*}
\phi_f(x_e) =
\left\{   
\begin{array}{ll}
       x_e  & \text{ if }  e\neq f\\
        x_{e_1} + x_{e_2} & \text{ if }  e = f
    \end{array}
    \right.
\end{equation*}

\begin{lemma}
\label{lem:2valentVerticesWCT}
With the above notation, we have
\begin{equation}
\label{eq:subdivideEdgeDelta}
(1\otimes \phi_f) \circ (\delta_{G}-I) = (\delta_{G'}-I) \circ (1\otimes \phi_f)       
\end{equation}
as morphisms $(L/H)_{R[G]} \to (L/H)_{R[G']}$, and 
\begin{equation*}
(1 \otimes \phi_f) (\bv_{\tau}(G)) = \bv_{\tau}(G').      
\end{equation*}
\end{lemma}

\begin{proof}
First, observe that the loops $\ell_{f}$, $\ell_{e_1}$, and $\ell_{e_2}$ are isotopic to each other, thus
\begin{equation*}
    T_{\ell_f} = T_{\ell_{e_1}} = T_{\ell_{e_2}}
\end{equation*}
as elements of $\Mod(\Sigma_g)$, and as a consequence, 
\begin{equation*}
    \delta_{\ell_f} = \delta_{\ell_{e_1}} = \delta_{\ell_{e_2}}.
\end{equation*}
As morphisms $H_{R[G]} \to H_{R[G']}$, we have
\begin{equation*}
    (1\otimes \phi_f) \circ (\delta_{e}-I) = \left\{
    \begin{array}{ll}
        (\delta_{\ell_e}-I) \circ (1\otimes \phi_f) & \text{if } e \neq f  \\
        ((\delta_{\ell_{e_1}}-I) + (\delta_{\ell_{e_2}}-I) ) \circ (1\otimes \phi_f) & \text{if } e = f
    \end{array}  \right.
\end{equation*}
Therefore,
\begin{align*}
    (1\otimes \phi_f) \circ (\delta_{G}-I) 
    &= (1\otimes \phi_f) \circ \sum_{e\in E(G)} (\delta_{\ell_{e}}-I) \\
    &= \left((\delta_{\ell_{e_1}}-I) + (\delta_{\ell_{e_2}}-I) + \sum_{e\in E(G)\setminus f} (\delta_{\ell_{e}}-I) \right) \circ (1\otimes \phi_f) \\
    &= (\delta_{G'}-I) \circ (1\otimes \phi_f)       
\end{align*}
Thus, the identity in Formula \ref{eq:subdivideEdgeDelta} holds as morphisms $H_{R[G]} \to H_{R[G']}$. One readily extends this as an identity of morphisms $(L/H)_{R[G]} \to (L/H)_{R[G']}$ using Proposition \ref{prop:fMinusIOnL}.

For the statement regarding the Ceresa class, we compute 
\begin{align*}
    (1\otimes \phi_f)(\bv_{\tau}(G)) &= J([T_{\ell_f},\tau]) \otimes (x_{e_1} + x_{e_2}) + \sum_{e\in E(G) \setminus f} J([T_{\ell_e},\tau]) \otimes x_e  \\
    &= \sum_{e\in E(G')} J([T_{\ell_e},\tau]) \otimes x_e \\
    &= \bv_{\tau}(G'). \qedhere
\end{align*}
\end{proof}

\begin{proposition}
\label{prop:stableWCT}
If $G$ and $G'$ are tropically equivalent graphs, then $G$ is Ceresa-Zharkov trivial if and only if $G'$ is Ceresa-Zharkov trivial. 
\end{proposition}

\begin{proof}
If $G'$ is obtained from $G$ by adding or contracting an edge adjacent to a 1-valent vertex, then the claim follows from Lemma \ref{lem:contractSeparatingEdge}. Now suppose $G'$ is obtained from $G$ by subdividing an edge $f \in E(G)$ into $e_1$ and $e_2$. If $G'$ is Ceresa-Zharkov trivial, then $G$ is Ceresa-Zharkov trivial by Proposition \ref{prop:weakCeresaContract}. Conversely, if  $G$ is Ceresa-Zharkov trivial, say $\bw_{\tau}(G) = 0$ for some $\tau$, then by Lemma \ref{lem:2valentVerticesWCT} we have
\begin{equation*}
    \bw_{\tau}(G')  = (\delta_{G}-I) \circ (1\otimes \phi_f) (\bv_{\tau}(G)) = (1\otimes \phi_f)(\bw_{\tau}(G)) = 0 
\end{equation*}
and hence $G'$ is Ceresa-Zharkov trivial. 
\end{proof}

\subsection{Relation to the tropical Ceresa class}

Let $\Gamma = (G,c)$ be a tropical curve. Recall from Section \ref{sec:tropicalCeresaClass} that $\bv(\Gamma) \in B(\delta_{\Gamma})$ is the image of the Ceresa class $\nu(\Gamma)$ under the map $A(\delta) \to B(\delta)$. Given $c:E(G) \to \Z_{>0}$, define homomorphisms 
\begin{align*}
    \epsilon_{c}^{1}: F_2(L/H) \otimes R_1 \to B(\delta_{\Gamma}) \hspace{20pt} \epsilon_{c}^{2}: F_3(L/H) \otimes R_2 \to C(\delta_{\Gamma}) 
\end{align*}
by the compositions
\begin{align*}
    &F_2(L/H)\otimes R_1 \to F_2(L/H) \to \gr_2^F(L/H) \to B(\delta_{\Gamma}),  \\
    &F_3(L/H)\otimes R_2 \to F_3(L/H) \to \gr_3^F(L/H) \to C(\delta_{\Gamma}),
\end{align*}
respectively, where the first maps $ F_2(L/H)\otimes R_1 \to F_2(L/H)$ and $F_3(L/H)\otimes R_2 \to F_3(L/H)$  are given by the evaluation map $h\otimes f \mapsto h\otimes f(c)$. We say that $\Gamma$ is \textit{Ceresa-Zharkov trivial} if $\bw(\Gamma) = 0$ in $C(\delta_{\Gamma})$.

\begin{proposition}
\label{prop:GtoGamma}
Let $\Gamma = (G,c)$ be a tropical curve, and let $\tau$ be a hyperelliptic quasi-involution such that $\bv_{\tau}(G)$ lies in $F_2(L/H)\otimes R_1$. Then
\begin{equation*}
 \epsilon_{c}^1(\bv_{\tau}(G)) = \bv(\Gamma), \hspace{15pt} \text{and} \hspace{15pt}   \epsilon_{c}^{2}(\bw_{\tau}(G)) = \bw(\Gamma).
\end{equation*}
In particular, if $G$ is Ceresa-Zharkov trivial, then so is $\Gamma$. 
\end{proposition}

\begin{proof}
This is a direct consequence of Formula \eqref{eq:CeresaInBdelta} and Proposition \ref{prop:sufficientWCT}. 
\end{proof}

By \cite[Proposition 4.6]{CEL}, if a tropical curve $\Gamma$ is hyperelliptic, then $\Gamma$ is Ceresa trivial. Assuming Theorem \ref{thm:WCTiffHET}, we derive a similar statement for tropical curves of hyperelliptic type and Ceresa-Zharkov triviality.

\begin{proposition}
\label{prop:HETthenWCT}
If a tropical curve $\Gamma$ is of hyperelliptic type, then $\Gamma$ is Ceresa-Zharkov trivial.
\end{proposition}
\begin{proof}
The statement follows from Proposition \ref{prop:GtoGamma} and Theorem \ref{thm:WCTiffHET}.
\end{proof}
\noindent However, the converse to Proposition \ref{prop:HETthenWCT} is not true by Remark \ref{rmk:K4NotWCT}.

\section{Ceresa-Zharkov trivial and tropical curves of hyperelliptic type}
\label{sec:MainTheorem}

The main goal of this section is to prove Theorem \ref{thm:WCTiffHET}, that a graph is Ceresa-Zharkov trivial if and only if it is of hyperelliptic type. 
Being of hyperelliptic type is a minor closed condition on graphs by \cite[Proposition 3.8]{Corey}. Because of this,  Theorem \ref{thm:WCTiffHET} implies that  Ceresa-Zharkov triviality is also a minor closed condition. Nevertheless, to prove this theorem, we need the fact that Ceresa-Zharkov triviality is preserved under edge contraction (Proposition \ref{prop:weakCeresaContract}) and under removal of an edge in a parallel pair. 

\begin{proposition}
\label{prop:removeParallelEdge}
	Consider either 
	\begin{enumerate}
	    \item a connected graph $G_1$ with a loop edge $a$, or
	    \item a 2-connected graph $G_2$ with a pair of parallel edges $(b,c)$.
	\end{enumerate}
	 If $G_1$, resp. $G_2$, is Ceresa-Zharkov trivial, then so is $G_1\setminus a$, resp.  $G_2\setminus b$.
\end{proposition}

\noindent 
We begin by describing, for any connected graph $G$, the images  $(\delta_{G}-I)(F_2(L_R/H_R))$ and $(\delta_{G}-I)^2(F_1(L_R/H_R))$. 
Order the edges of $G$ by $e_1,\ldots,e_n$ so that $e_{g+1},\ldots,e_n$ are the edges of a spanning tree. Let $\alpha_1,\ldots,\alpha_g,\beta_{1},\ldots,\beta_{g}$ be the basis of $H_1(\Sigma_g,\mathbb{Z})$ from \S  \ref{sec:graphs}, so $\delta_{G}$ has the form in \eqref{eq:deltaQ}. Write $q_{ij}$ for the entries of $Q_G$. 

\begin{lemma}\label{lem:image1}
We have
\begin{equation*}
(\delta_G-I) \left( \sum_{i; j<k} \abb{i}{j}{k}  \otimes b_{ijk} \right)= \sum_{r<s<t}  \bbb{r}{s}{t} \otimes c_{rst}.
\end{equation*}
where
\begin{align*}
    c_{rst} = \sum_{i} (b_{irs}q_{ti} - b_{irt}q_{si} + b_{ist}q_{ri}).
\end{align*}
\end{lemma}

\begin{proof}
By Formula \eqref{eq:deltaGQGformulas}, we have that
\begin{align*}
    (\delta_G-I)  \sum_{i; j<k} \abb{i}{j}{k} \otimes b_{ijk}  = \sum_{i; j<k}  Q_{G}(\alpha_i) \wedge \beta_j \wedge \beta_k \otimes b_{ijk} 
    =  \sum_{i; j<k; \ell}  \bbb{\ell}{j}{k} \otimes  b_{ijk} q_{\ell i}.
\end{align*}
Given $r<s<t$ one may extract the coefficient $c_{rst}$ for $\bbb{r}{s}{t}$. 
\end{proof}

\begin{lemma}\label{lem:image2}
We have
\begin{equation*}
(\delta_G-I)^2 \sum_{i<j; k} \aab{i}{j}{k} \otimes  a_{ijk} = \sum_{r<s<t}  \bbb{r}{s}{t} \otimes c_{rst} \hspace{10pt} \text{where} \hspace{5pt}    c_{rst} = 2 \sum_{i<j} \,
\begin{vmatrix}
q_{ri} & q_{rj} & a_{ijr} \\
q_{si} & q_{sj} & a_{ijs} \\
q_{ti} & q_{tj} & a_{ijt} 
\end{vmatrix}.
\end{equation*}
\end{lemma}

\begin{proof}
By Formula \ref{eq:deltaGQGformulas}, we have that
\begin{align*}
    (\delta_G-I)^2 \left( \sum_{i<j; k} \aab{i}{j}{k} \otimes  a_{ijk} \right) = \sum_{i<j; k} 2 Q_{G} \, \alpha_i \wedge Q_{G} \, \alpha_j \wedge \, \beta_k \otimes  a_{ijk}\\
    =  \sum_{i<j; k} \sum_{\ell<m} \bbb{\ell}{m}{k} \otimes  2a_{ijk}  
    \begin{vmatrix}
        q_{\ell i} & q_{\ell j} \\
        q_{m i} & q_{m j}
    \end{vmatrix}.
\end{align*}
 Given $r<s<t$, the $R$-coefficient of $\bbb{r}{s}{t}$ is
\begin{equation*}
    c_{rst} = 2\sum_{i<j}  \left(
    a_{ijr} \begin{vmatrix}
        q_{si} & q_{sj} \\
        q_{ti} & q_{tj}
    \end{vmatrix} - 
    a_{ijs} \begin{vmatrix}
        q_{ri} & q_{rj} \\
        q_{ti} & q_{tj}
    \end{vmatrix} +
    a_{ijt} \begin{vmatrix}
        q_{ri} & q_{rj} \\
        q_{si} & q_{sj}
    \end{vmatrix}
     \right)
\end{equation*}
and the summand is exactly the $3\times 3$ determinant appearing in the lemma.
\end{proof}

\begin{figure}[tbh!]
    \centering
    \includegraphics[width=\textwidth]{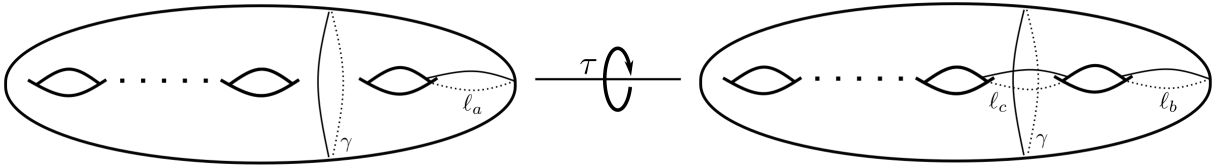}
    \caption{Arrangements of curves on $\Sigma_{g+1}$ with dual graphs $G_1$ and $G_2$ where (left) $a$ is a loop edge and (right) $(b,c)$ are parallel edges. Here, we identify $\Sigma_{g}^1$ with the subsurface of $\Sigma_{g+1}$ to the left of $\gamma$. }
    \label{fig:SigmagSigmag1}
\end{figure}

Our next step is to show how the Ceresa classes of $G$ and $G \setminus f$ are related. 
The two cases listed in Proposition \ref{prop:removeParallelEdge} are similar, so we handle them in parallel. Let $G_1$ be a connected graph with a loop edge $a$ and let $G_2$ be a 2-connected graph with a pair of parallel edges $b,c$. Let $\Lambda_1 = \{\ell_e \, : \, e\in E(G_1)\}$ and $\Lambda_2 = \{\ell_e \, : \, e\in E(G_2)\}$ be two arrangements of isotopy classes of simple closed curves whose dual graphs are $G_1$ and $G_2$ respectively.
Suppose that $\ell_{a}$, $\ell_{b}$, $\ell_{c}$ are as in Figure \ref{fig:SigmagSigmag1}, and every other $\ell_{e}$ lies in $\Sigma_{g}^1$, the genus $g$ subsurface with a boundary to the left of $\gamma$. 
To emphasize the dependence on the genus $g$, we write
\begin{equation*}
    H(\Sigma_g^{n}) = H_1(\Sigma_{g}^{n},\Z), \hspace{10pt} L(\Sigma_{g}^{n}) = \wedge^3H_1(\Sigma_g^{n}).
\end{equation*}
where $n=0$ or $1$. 
The inclusion $\Sigma_{g}^1\hookrightarrow \Sigma_{g+1}$ induces an inclusion on homology groups which we use to identify  $H(\Sigma_g) \cong H(\Sigma_g^1)$ as a subgroup of $H(\Sigma_{g+1})$. 
Now, recall that we obtain $\Sigma_{g}$ from $\Sigma_{g}^1$ by attaching a disc along $\gamma$. Any two extensions of $\ell_{c}\cap \Sigma_{g}^1$ to $\Sigma_{g}$ are isotopic. Choose such an extension and denote it also by $\ell_{c}$. 

We may view $\Lambda_1\setminus \{\ell_a\}$ as an arrangement of curves on $\Sigma_g$ or $\Sigma_g^1$. We may also view $\Lambda_2\setminus \{\ell_b\}$ as an arrangement of curves on $\Sigma_g$, and  $\Lambda_2\setminus \{\ell_b,\ell_{c}\}$ as an arrangement of curves on $\Sigma_g^1$. Because $L(\Sigma_{g})/H(\Sigma_g)$ does not naturally  embed into $L(\Sigma_{g+1})/H(\Sigma_{g+1})$, we cannot directly compare
$\bv_{\tau}(\Lambda_1)$ with $\bv_{\tau'}(\Lambda_1\setminus \{\ell_a\})$, or $\bv_{\tau}(\Lambda_2)$ with $\bv_{\tau'}(\Lambda_2\setminus \{\ell_b\})$. Instead, we compare these with Ceresa classes on $\Sigma_g^1$ in the following way. 

\begin{lemma}
\label{lem:mu2nu}
There are hyperelliptic quasi-involutions  $\tau$ on $\Sigma_{g+1}$, $\tau'$ on $\Sigma_{g}$, and $\tau''$ on $\Sigma_{g}^1$, under the natural homomorphisms 
\begin{align*} 
    &L(\Sigma_g^1)_{R[\Lambda_1 \setminus \{\ell_a\}]}  \to L(\Sigma_g)_{R[\Lambda_1 \setminus \{\ell_a\}]} / H(\Sigma_g)_{R[\Lambda_1 \setminus \{\ell_a\}]} \\
    &L(\Sigma_g^1)_{R[\Lambda_1 \setminus \{\ell_a\}]} \hookrightarrow L(\Sigma_{g+1}^1)_{R[\Lambda_1]} \to L(\Sigma_{g+1})_{R[\Lambda_1]} / H(\Sigma_{g+1})_{R[\Lambda_1]}
\end{align*}
the class $\bmu_{\tau''}(\Lambda_1\setminus \{\ell_a\})$ maps to $\bv_{\tau'}(\Lambda_1\setminus \{\ell_a\})$ and $\bv_{\tau}(\Lambda_1)$, respectively. Similarly, the natural homomorphisms
\begin{align*}
    L(\Sigma_g)_{R[\Lambda_2\setminus \{\ell_{b},\ell_{c}\}]} &\to L(\Sigma_g)_{R[\Lambda_2 \setminus \{\ell_b\}]}/H(\Sigma_g)_{R[\Lambda_2 \setminus \{\ell_{b}\}]} \\ 
    L(\Sigma_g^1)_{R[\Lambda_2\setminus \{\ell_b,\ell_{c}\}]} &\hookrightarrow L(\Sigma_{g+1})_{R[\Lambda_2]} \to  L(\Sigma_{g+1})_{R[\Lambda_2]}/H(\Sigma_{g+1})_{R[\Lambda_2]}
\end{align*}
take $\bmu_{\tau''}(\Lambda_2\setminus \{\ell_b,\ell_c\})$  to $\bv_{\tau'}(\Lambda_2\setminus \{\ell_b\})$ and $\bv_{\tau}(\Lambda_2)$, respectively.

Furthermore, we may choose  $\tau$, $\tau'$, and $\tau''$ so that all classes defined above live in the $F_2$ part of the relevant filtration.
\end{lemma}

\begin{proof}
Consider the surfaces homeomorphic to $\Sigma_{g+1}$ in Figure  \ref{fig:SigmagSigmag1}. These depict $\Sigma_{g+1}$ embedded in $\R^3$ and bounding a handlebody $V$, the ``inside'' of the surface. Assume that the curves in $\Lambda_1$ and $\Lambda_2$ are \textit{meridians}, i.e., they bound properly embedded discs in $V$. We use these assumptions in the next paragraph to show that the Ceresa classes all belong to $F_2$.  

Next, let $\tau$ be the hyperelliptic involution on $\Sigma_{g+1}$ given by rotation by $180^{\circ}$ about the axis illustrated in Figure \ref{fig:SigmagSigmag1}; in particular, $\tau$ takes $\ell_f$ to $\ell_f$ (reversing its orientation) for $f=a$, $b$, $c$. The map $\tau$ can be isotoped in a regular neighborhood of $\gamma$ so that it fixes $\gamma$ pointwise. Consequently, the mapping class $\tau \in \Mod(\Sigma_{g+1})$ restricts to a hyperelliptic quasi-involution $\tau''$ on $\Sigma_{g}^{1}$; denote by $\tau'$ the image of $\tau''$ under the natural map $\Mod(\Sigma_{g}^1) \to \Mod(\Sigma_{g})$. Because all the curves in $\Lambda_1$ and $\Lambda_2$ are meridians and the hyperelliptic quasi-involutions $\tau$, $\tau'$, and $\tau''$ extend to homeomorphisms on the handlebodies,  each $J([T_{\ell}, \tau''])$, $J([T_{\ell}, \tau'])$, and $J([T_{\ell}, \tau])$ lie in $F_2$ by \cite[Theorem~6.6]{CEL}, and therefore $\bmu_{\tau''}$, $\bv_{\tau'}$, and $\bv_{\tau'}$ lie in $F_2\otimes R_1$. 

Since $[T_{\ell_{f}},\tau] = T_{\ell_{f}}T_{\tau(\ell_f)}^{-1}$ and $\tau(\ell_f)$ is isotopic to $\ell_f$ for $f=a,b,c$, we have that $[T_{\ell_f},\tau] = 1$ for $f=a,b,c$. Similarly, we have $[T_{\ell_c},\tau'] = 1$. Therefore, 
{\footnotesize
\begin{equation*}
    \bmu_{\tau''}(\Lambda_1\setminus \{\ell_a\}) =  \sum_{e\in E(G_1)\setminus \{a\}} J([T_{\ell_e}, \tau ]) \otimes x_e \hspace{10pt} \text{and} \hspace{10pt} \bmu_{\tau''}(\Lambda_2\setminus \{\ell_b,\ell_c\}) =  \sum_{e\in E(G_1)\setminus \{b,c\}} J([T_{\ell_e}, \tau ]) \otimes x_e
\end{equation*}
}
The lemma follows from these computations. 
\end{proof}

\begin{proof}[Proof of Proposition \ref{prop:removeParallelEdge}] 
Suppose that the graph $G$ and edge $f\in E(G)$ are either 
\begin{enumerate}
    \item $G=G_1$  and $f=a$ is a loop edge, or 
    \item $G=G_2$, a 2-connected graph, and $f=b$ is parallel to another edge $c$. 
\end{enumerate}
Let $g+1$ be the genus of $G$. Order the edges of $G$ by  $e_0, \ldots, e_g, e_{g+1}, \ldots, e_n$ such that the last $n-(g+1)$ edges form a spanning tree of $G$ and $e_0 = f$. In the second case, assume that $e_1=c$ and orient $[\ell_{b}]$ and $[\ell_{c}]$ in the same direction. This is possible since $G_2$ is 2-connected, so $b,c$ are not a separating pair.  Denote by $\alpha_{0},\ldots,\alpha_g,\beta_0,\ldots,\beta_g$ the basis described in Section \ref{sec:graphs}. This identifies $Q_{G\setminus f}$ with the lower-right $g\times g$ submatrix of $Q_{G}$. Let $q_{ij}$ denote the coordinates of $Q_{G}$. When $G=G_1$ we have
\begin{equation*}
    q_{00} = x_{0}, \hspace{10pt}  \text{and} \hspace{10pt} q_{0j} = q_{j0} = 0 \hspace{10pt} \text{ for } \hspace{10pt} j\geq 1 
\end{equation*}
and when $G=G_2$, we have
\begin{equation*}
    q_{00} = x_{0} + q_{01},\hspace{10pt} q_{11} = x_{1} + q_{01} \hspace{10pt}  \text{and} \hspace{10pt} q_{0j} = q_{1j} \hspace{10pt} \text{ for } \hspace{10pt} j\geq 2.
\end{equation*}

By Lemma \ref{lem:mu2nu}, there are hyperelliptic quasi-involutions $\tau$ on $\Sigma_{g+1}$ and $\tau'$ on $\Sigma_{g}$ such that 
\begin{align*}
\bv_{\tau}(G) = \sum_{
\substack{1\leq i\leq g \\ 1\leq j<k\leq g}}  \abb{i}{j}{k} \otimes b_{ijk} \hspace{10pt} \text{ in } \hspace{10pt} L(\Sigma_{g+1})_{R[G]}/H(\Sigma_{g+1})_{R[G]} \\
\bv_{\tau'}(G\setminus f) = \sum_{\substack{1\leq i\leq g\\ 1\leq j<k\leq g}}  \abb{i}{j}{k} \otimes b_{ijk} \hspace{10pt} \text{ in } \hspace{10pt} L(\Sigma_{g})_{R[G\setminus f]}/H(\Sigma_{g})_{R[G\setminus f]}
\end{align*}
where $b_{ijk}$'s are linear forms in $R[G\setminus f]\subset R[G]$. Implicit in this description is that $b_{0jk} = 0$ and $b_{i0k}=0$. Because $G$ is Ceresa-Zharkov trivial, by Proposition \ref{prop:sufficientWCT}, there is a $\bu\in F_1(L/H)$, say
\begin{equation*}
\bu= \sum_{\substack{0\leq i<j \leq g \\ 0\leq k \leq g}} a_{ijk} \cdot \aab{i}{j}{k} 
\end{equation*}
with $a_{ijk} \in \mathbb{Z}$ such that $\bw_{\tau}(G) = (\delta_G-I)^2(\bu)$. We will show that $\bw_{\tau'}(G\setminus f)$ lies in the image $(\delta_{G\setminus f}-I)^2(F_1(L(\Sigma_g)/H(\Sigma_g)))$, whence $G\setminus f$ is Ceresa-Zharkov trivial by Proposition \ref{prop:sufficientWCT}. Because $b_{0jk} = 0$, we have
\begin{equation}
\label{eq:crstCeresa}
     c_{rst} := \sum_{i=0}^{g} (b_{irs}q_{ti} - b_{irt}q_{si} + b_{ist}q_{ri}) = \sum_{i=1}^{g} (b_{irs}q_{ti} - b_{irt}q_{si} + b_{ist}q_{ri})
\end{equation}
for $0\leq r<s<t \leq g$. Therefore, by Lemma \ref{lem:image1}, we have
\begin{align*}
    \bw_{\tau}(G) &= (\delta_G-I)(\bv_{\tau}(G)) = \sum_{0\leq r<s<t \leq g} \bbb{r}{s}{t} \otimes c_{rst} \hspace{15pt} \text{and} \\   \bw_{\tau'}(G\setminus f) &= (\delta_{G\setminus f}-I)(\bv_{\tau'}(G\setminus f)) = \sum_{1\leq r<s<t \leq g} \bbb{r}{s}{t} \otimes c_{rst}.
\end{align*}
By the equality $\bw_{\tau}(G) = (\delta_{G}-I)^2(\bu)$ and Lemma \ref{lem:image2}, we have
\begin{align}
\label{eq:crstdeltaG}
    c_{rst} = 2\sum_{0\leq i<j \leq g}
\begin{vmatrix}
q_{ri} & q_{rj} & a_{ijr} \\
q_{si} & q_{sj}  & a_{ijs} \\
q_{ti} & q_{tj}  & a_{ijt} 
\end{vmatrix}.
\end{align}
Set
\begin{equation*}
        \bu'= \sum_{\substack{1\leq i<j \leq g \\ 1\leq k \leq g}} a_{ijk} \cdot \aab{i}{j}{k}.
\end{equation*}
Then 
\begin{equation*}
    (\delta_{G\setminus f}-I)^2(\bu') = \sum_{1\leq r<s<t \leq g}  \bbb{r}{s}{t} \otimes c_{rst}' \hspace{15pt} \text{where} \hspace{15pt} c_{rst}' = 2\sum_{1\leq i<j \leq g}
\begin{vmatrix}
q_{ri} & q_{rj} & a_{ijr} \\
q_{si} & q_{sj}  & a_{ijs} \\
q_{ti} & q_{tj}  & a_{ijt} 
\end{vmatrix}. 
\end{equation*}

Next let's compute the difference 
\begin{align*}
&\bw_{\tau'}(G\setminus f) -
(\delta_{G\setminus f}-I)^2(\bu')  = \sum_{1\leq r<s<t \leq g} \bbb{r}{s}{t} \otimes (c_{rst} - c_{rst}')
\end{align*}
where
\begin{equation*}
    c_{rst} - c_{rst}' =  2\sum_{j=1}^g
\begin{vmatrix}
q_{r0} & q_{rj} & a_{0jr} \\
q_{s0} & q_{sj}  & a_{0js} \\
q_{t0} & q_{tj}  & a_{0jt} 
\end{vmatrix}
\hspace{10pt} \text{for} \hspace{10pt} 
1\leq r < s < t \leq g.
\end{equation*}
If we are in the first case, i.e., $G=G_1$ and $f=a$ is a loop edge, then $q_{j0}=0$ for $j\geq 1$, so the above formula implies $c_{rst}-c_{rst}' = 0$, whence $\bw_{\tau'}(G\setminus f) = (\delta_{G\setminus f}-I)^2(\bu')$. 

For the rest of the proof, suppose we are in the second case, i.e., $G=G_2$ and $f=b=e_0$ is parallel to the edge $c=e_1$. Set
\begin{equation*}
 \bu''= \sum_{\substack{2\leq j \leq g\\ 1\leq k \leq g}} a_{0jk} \cdot \aab{1}{j}{k}.
\end{equation*}
We have
\begin{align*}
    (\delta_{G\setminus f} -I)^2(\bu'') = \sum_{1\leq r<s<t \leq g} \bbb{r}{s}{t} \otimes  c_{rst}'' \hspace{15pt} \text{where} \hspace{15pt} c_{rst}'' = 2 \sum_{j=2}^g 
    \begin{vmatrix}
    q_{r1} & q_{rj} & a_{0jr} \\
    q_{s1} & q_{sj} & a_{0js} \\
    q_{t1} & q_{tj} & a_{0jt} 
    \end{vmatrix}.
\end{align*}
When $r\geq 2$,  we have that $q_{r0} = q_{r1}$, and therefore $c_{rst} - c_{rst}' = c_{rst}''$. Now consider $r=1$. 
Since $x_0$ only appears in $q_{00}$ where $q_{00} = x_{0} + q_{01}$, the coefficient of $x_0$ in $c_{rst}$ must equal 0 for all $0\leq r<s<t\leq g$ by Formula \eqref{eq:crstCeresa}. By extracting the coefficient of $x_0$ in the expression of $c_{0st}$ from Formula  \eqref{eq:crstdeltaG},  we see that
\begin{equation}
\label{eq:coefx0}
    \sum_{j=1}^{g}  
    \begin{vmatrix}
 q_{sj}  & a_{0js} \\
 q_{tj}  & a_{0jt} 
\end{vmatrix} = 0. 
\end{equation}
Therefore,
\begin{align*}
    c_{1st}'' - (c_{1st} - c_{1st}') = 2 (q_{11}-q_{10}) \sum_{j=1}^g \begin{vmatrix}
    q_{sj} & a_{0js} \\
    q_{tj} & a_{0jt} 
    \end{vmatrix} =0
\end{align*}
where the last equality follows from Formula \eqref{eq:coefx0}. So $c_{rst} = c_{rst}' + c_{rst}''$ and therefore
\begin{align*}
    \bw_{\tau'}(G\setminus f) &=\sum_{1\leq r<s<t \leq g} \bbb{r}{s}{t} \otimes  c_{rst} \\
    &= \sum_{1\leq r<s<t \leq g} \bbb{r}{s}{t} \otimes  (c_{rst}' + c_{rst}'')
    = (\delta_{G\setminus f}-I)^2(\bu' + \bu''). \qedhere
\end{align*}
\end{proof}

Next, we show that Ceresa-Zharkov triviality may be detected at the level of $2$-connected components.

\begin{proposition}
\label{prop:WCT2connectedComponents}
A connected graph $G$ is Ceresa-Zharkov trivial if and only if its 2-connected components are Ceresa-Zharkov trivial. 
\end{proposition}

\begin{proof}
By induction, it suffices to consider the case where $G$ has two 2-connected components  $G_1$ and $G_2$. Choose a separating curve $\gamma$ and an arrangement of simple closed curves $\Lambda = \{\ell_e \, : \, e\in E(G) \}$ such that
\begin{itemize}
    \item[-] the dual graph of $\Lambda$ is $G$;
    \item[-] cutting along $\gamma$ separates $\Sigma_g$ into two subsurfaces $S_1\cong \Sigma_{g_1}^1$ and $S_2 \cong \Sigma_{g_2}^1$;
    \item[-] $\ell_e$ lies in $S_i$ whenever $e\in E(G_i)$. 
\end{itemize}
Denote by $H^{(i)} = H_1(\Sigma_{g_i}, \Z)$ and  $L^{(i)} = \wedge^3H^{(i)}$. The inclusion $S_i\subset \Sigma_g$ allows us to view $F_qL^{(i)} \otimes R$  as an $R$-submodule of $F_qL_R$ for $q=0,\ldots,3$.
Choose hyperelliptic quasi-involutions $\tau_1'$ of $S_1$ and $\tau_2'$ of $S_2$ such that $\bmu_{\tau_i'}(\Lambda_i')$ lies in $F_2L^{(i)} \otimes R_1$; this is possible by Proposition \ref{prop:CeresaF2R1}.  Let $\tau$ be the hyperelliptic quasi-involution of $\Sigma_g$ obtained that restricts to $\tau_i'$ on $S_i$.  Denote by $\tau_i\in \Mod(\Sigma_{g_i})$ the image of $\tau'$ under the natural map $\Mod(\Sigma_{g_i}^1) \to \Mod(\Sigma_{g_i})$; thus $\tau_i$ is a hyperelliptic quasi-involution of $\Sigma_{g_i}$. Then  $\bv_{\tau_1}(G_1), \bv_{\tau_2}(G_2), \bv_{\tau}(G)$ are in $F_2L\otimes R_1$, and therefore  $\bw_{\tau_1}(G_1), \bw_{\tau_2}(G_2), \bw_{\tau}(G)$ are in $F_3L\otimes R_2$.

Consider the following commutative diagram:
\begin{center}
\begin{equation}
\label{eq:L1L2L}
\begin{tikzcd} 
(F_2L^{(1)} \oplus F_2L^{(2)})_R \arrow[hookrightarrow]{r} \arrow[d,swap,"(\delta_{G_1}-I) \oplus (\delta_{G_2}-I)"]& F_2L_R \arrow[r] \arrow[d,"\delta_G-I"] & F_2(L/H)_R \arrow[d,"\delta_G-I"] \\
(F_3L^{(1)} \oplus F_3L^{(2)})_R \arrow[hookrightarrow]{r} & F_3L_R \arrow[r,"\sim"] & F_3(L/H)_R 
\end{tikzcd}
\end{equation}
\end{center}
The right arrow on the bottom row is an isomorphism by Formula \eqref{eq:explicitFq}. Consider $\bmu := \bmu_{\tau_1'}(\Lambda_1') + \bmu_{\tau_2'}(\Lambda_2')$ in $(F_2L^{(1)} \oplus F_2L^{(2)})_R$. The composition of the top two arrows of the diagram in \eqref{eq:L1L2L} maps $\bmu$ to $\bv_{\tau}(G)$, which maps to $\bw_{\tau}(G)$ by the right vertical arrow. The composition along the bottom takes $\bmu$ to $\bw_{\tau_1}(G_1) + \bw_{\tau_2}(G_2)$. Thus
\begin{equation}
\label{eq:tau1tau2eqtau}
\bw_{\tau_1}(G_1) + \bw_{\tau_2}(G_2) = \bw_{\tau}(G).
\end{equation}

Suppose $G_1$ and $G_2$ are Ceresa-Zharkov trivial. By Proposition \ref{prop:sufficientWCT}, there are elements $\bu_i\in F_1(L^{(i)}/H^{(i)})$ such that $(\delta_{G_i}-I)^2(\bu_i) = \bw_{\tau_i}(G_i)$.  The restriction of  $(\delta_G-I)^2$ to $F_1L^{(1)} \oplus F_1L^{(2)}$ is $(\delta_{G_1}-I)^2 \oplus (\delta_{G_2}-I)^2$; this follows from Formula \eqref{eq:deltaGQGformulas} and the fact that $Q_G = Q_{G_1}\oplus Q_{G_2}$. So  $(\delta_{G}-I)^2(\bu_1+\bu_2) = \bw_{\tau}(G)$, and therefore $G$ is Ceresa-Zharkov trivial. 

Conversely, suppose $G$ is Ceresa-Zharkov trivial. Let $T$ be a spanning tree of $G_1$. Then $G/T$  is Ceresa-Zharkov trivial by Proposition \ref{prop:weakCeresaContract}. Observe that $G/T$ is obtained from $G_2$ by attaching $g_1$ loop edges to a single vertex of $G_2$. So $G_2$ is Ceresa-Zharkov trivial by Proposition \ref{prop:removeParallelEdge}.  Swapping the roles of $G_1$ and $G_2$, we conclude that $G_1$ is also Ceresa-Zharkov trivial. 
\end{proof}

Recall from \cite[\S 3]{Corey} that a graph $G$ is \textit{strongly of hyperelliptic type} if there is a choice of edge lengths of $G$ so that the resulting tropical curve is hyperelliptic. Such graphs are Ceresa-Zharkov trivial by the following proposition. 

\begin{proposition}
\label{prop:tauSwapSeparatingPair}
Let $G$ be a graph that has a separating pair of edges $(f,f')$, and let $\Lambda = \{\ell_e \, : \, e\in E(G)\}$ be a collection of pairwise disjoint simple closed curves on $\Sigma_g$ with dual graph $G$. If $\tau$ is a hyperelliptic quasi-involution such that $\tau(\ell_{f}) = \ell_{f'}$, then
\begin{equation*}
    (\delta_{G}-I)(J([T_{\ell_{f}}, \tau])\otimes 1)  =  0.
\end{equation*}
In particular, if $G$ is strongly of hyperelliptic type, then $G$ is Ceresa-Zharkov trivial.
\end{proposition}

\begin{proof}
The removal of $\{f,f'\}$ from $G$ separates this graph into two graphs $G_1$ and $G_2$ of gerena $g_1$ and $g_2$, respectively. Order the edges of $G$ by $e_1, \ldots, e_g, e_{g+1}, \ldots, e_{n}$ so that
\begin{itemize}
    \item[-] $e_1,\ldots,e_{g_1} \in E(G_1)$, $e_{g_1+1},\ldots,e_{g-1} \in E(G_2)$, and $e_{g} = f$; 
    \item[-] $e_{g+1}, \ldots, e_{n}$ are the edges of a spanning tree of $G$.
\end{itemize}
Necessarily $f'$ must be among the edges of the spanning tree. Let $\alpha_1,\ldots,\alpha_g,\beta_1,\ldots,\beta_g$ be the basis from \S \ref{sec:graphs}. With respect to this basis, we have
\begin{equation*}
    Q_G = \begin{bmatrix}
    Q_{G_1} & 0 & * \\
    0 & Q_{G_2} & * \\
    * & * & x_{f} 
    \end{bmatrix}.
\end{equation*}
Cutting along $\ell_{f}$ and $\ell_{f'}$ separates the surface $\Sigma_g$ into two subsurfaces $S_1$ and $S_2$. The homology of $S_i$ splits as a direct sum $V\oplus W_i$ where $V = \Z \cdot [\ell_f]$ and $W_i$ is identified with a symplectic subspace of $H$ under the map $H_1(S_i,\Z) \to H$. The intersection 2-form $\omega_i$ of $W_i$ is the restriction of $\omega$ to $W_i$, so
\begin{equation*}
    \omega_1 = \sum_{j=1}^{g_1} \alpha_j \wedge \beta_j \hspace{15pt} \text{and}  \hspace{15pt} \omega_2 = \sum_{j=g_1+1}^{g-1} \alpha_j \wedge \beta_j.
\end{equation*}
Orient $[\ell_{f}]$ such that $S_1$ lies to its right. Because $\tau(\ell_{f}) = \ell_{f'}$, we have that $\tau T_{\ell_{f}} \tau^{-1} = T_{\ell_{f'}}$, and hence
\begin{equation*}
    J([T_{\ell_{f}}, \tau]) = J(T_{\ell_{f}} T_{\ell_{f'}}^{-1}) = \omega_1\wedge [\ell_{f}] = - \omega_2 \wedge [\ell_{f'}]. 
\end{equation*}
The expression $(\delta_{G}-I)(J(T_{\ell_{f}} T_{\ell_{f'}}^{-1}))$ is equal to
\begin{equation*}
    (\delta_{G}-I)(\omega_1\wedge [\ell_{f}]) = \sum_{j=1}^{g_1} (Q_{G} \, \alpha_j ) \wedge \beta_j \wedge \beta_g = \sum_{j=1}^{g_1} (Q_{G_1} \, \alpha_j ) \wedge \beta_j \wedge \beta_g  = (\delta_{G_1}-I)(\omega_1\wedge [\ell_{f}]).
\end{equation*}
We use Formula \eqref{eq:deltaGQGformulas} in the first equality and $Q_{G}(\alpha_j) -  Q_{G_1}(\alpha_j) \in \Z\cdot \beta_g$ for $j\leq g_1$ in the second.
By Proposition \ref{prop:BonOmega} applied to $V$ and $W=W_1$, we have
\begin{equation*}
    (\delta_{G_1}-I)(\omega_1\wedge [\ell_{f}]) = \omega_1\wedge (\delta_{G_1}-I) ([\ell_{f}]) + \eta \wedge \delta_{G_1} ([\ell_f]).
\end{equation*}
for some $\eta \in V\wedge W_1$.  Since the curves in $\Lambda$ are disjoint, we have that $\delta_{G_1}([\ell_f]) = [\ell_{f}]$. This implies that the first summand above is 0, and the second summand is in $V\wedge W_1\wedge V$, which must also be 0 since $\dim V = 1$. So we have $(\delta_{G}-I)(J([T_{\ell_{f}}, \tau])) = 0$, as required. 

Now suppose $G$ is strongly of hyperelliptic type. By Proposition \ref{prop:stableWCT}, we may assume that $G$ is stable. Let $\Gamma$ be a hyperelliptic tropical curve with underlying graph $G$, and $\sigma$ be a hyperelliptic involution of $\Gamma$. There is a hyperelliptic quasi-involution $\tau$ of $\Sigma_g$ such that $\tau(\ell_{e}) = \ell_{\sigma(e)}$ by \cite[Lemma 4.5]{CEL}.  By \cite[Proposition~2.5]{Corey}, for any edge $e\in E(G)$, we have that $\sigma(e) = e$ or $\sigma(e) = f$ where $(e,f)$ is a separating pair. In the first case, $[T_{\ell_e},\tau] = 1$. In the second case, $(\delta_G-I)(J([T_{\ell_e},\tau])) = 0$ by the first part of this proposition.  We conclude that $\bw_{\tau}(G) = 0$, i.e., $G$ is Ceresa-Zharkov trivial. 
\end{proof}

To prove Ceresa-Zharkov trivial implies hyperelliptic type, we use in an essential way the main theorem of \cite{Corey}, which states that a graph is of hyperelliptic type if and only if it has no $K_4$ or $L_3$ minor. We prove directly that these graphs are not Ceresa-Zharkov trivial, using Proposition \ref{prop:GtoGamma}.

\begin{proposition}
\label{prop:K4Nontrivial}
The graph $G=K_4$ is not Ceresa-Zharkov trivial.
\end{proposition}

\begin{proof}
By \cite[Example 7.2]{CEL}, a Ceresa cocycle $\bv_{\tau}(G)$ for the graph $K_4$ is given by
\begin{equation}
\label{eqn:muK4}
\bv_{\tau}(G) =  \abb{1}{1}{2} \otimes x_2 + (-\abb{2}{1}{2} - \abb{2}{2}{3} + \abb{2}{1}{3}) \otimes x_5.
\end{equation}
which lies in $F_2(L/H) \otimes R_1$.  The matrix $Q_G$ is recorded in Formula \eqref{eq:QK4}. We get
\begin{equation*}
\bw_{\tau}(G) = -2 \cdot \bbb{1}{2}{3} \otimes x_2x_5. 
\end{equation*}
which lies in $F_3(L/H) \otimes R_2$. By Proposition \ref{prop:GtoGamma}, it is sufficient to show that the tropical curve $\Gamma = (K_4,c)$ is not Ceresa-Zharkov trivial for some edge-length function $c:E(K_4) \to \Z_{>0}$. Let $c$ be the function that assigns the length 1 to each edge.  Then $\bw(\Gamma) = -2 \, \bbb{1}{2}{3}$, whereas $(\delta_{\Gamma}-I)^2(F_1(L/H))$ is spanned by $4\, \bbb{1}{2}{3}$, see \cite[Remark~7.3]{CEL}. So $\bw(\Gamma) \neq 0$, as required.
\end{proof}

\begin{remark}
\label{rmk:K4NotWCT}
    Define $c:E(K_4) \to \Z_{>0}$ by $c(e_1) = 2$ and $c(e_i)=1$ for $i=2,\ldots,6$ and let $\Gamma=(K_4,c)$. Then $(\delta_{\Gamma}-I)^2(F_1(L/H)) = F_3L$, and hence $\bw(\Gamma) = 0$. Therefore $\Gamma$ is Ceresa-Zharkov trivial, but clearly not of hyperelliptic type. 
\end{remark}

\begin{proposition}\label{prop:L3Nontrivial}
The graph $G=L_3$ is not Ceresa-Zharkov trivial.
\end{proposition}

\begin{proof}
From \cite[Example 7.6]{CEL} by setting $c_7=c_8=c_9=0$, we get

\begin{equation*}
    Q_G = \begin{pmatrix}
    x_1 +  x_6 & 0 & x_6 & x_6 \\
    0 & x_2+x_5 & x_5 & x_5 \\
    x_6 & x_5 & x_3+x_5+x_6 & x_5+x_6 \\
    x_6 & x_5 & x_5+x_6 & x_4+x_5+x_6
    \end{pmatrix}.
\end{equation*}

\begin{align*}
\bv_{\tau}(G)= &(\abb{2}{2}{3} + \abb{2}{2}{4}-\abb{2}{1}{2}) \otimes x_6 \\
-&(\abb{1}{1}{2}+ \abb{1}{1}{3} + \abb{1}{1}{4})\otimes x_5.
\end{align*}
So we have
\begin{equation*}
    \bw_{\tau}(G) = -2x_5x_6(\beta_1\wedge\beta_2\wedge\beta_3+\beta_1\wedge\beta_2\wedge\beta_4).
\end{equation*}
By Proposition \ref{prop:GtoGamma}, it is sufficient to show that the tropical curve $\Gamma = (L_3,c)$ is not Ceresa-Zharkov trivial for some edge-length function $c:E(L_3) \to \Z_{>0}$. Let $c$ be the function that assigns the length 1 to each edge. Using \texttt{sage}, we see that  $(\delta_{\Gamma}-I)^2(F_1(L/H))$ is spanned by 
\begin{align*}
\begin{array}{ll}
     2\bbb{1}{2}{3} + 2\bbb{1}{2}{4} + 2\bbb{2}{3}{4}, & \; 4\bbb{1}{2}{4}, \\
    2\bbb{1}{3}{4} +2\bbb{2}{3}{4}, & 4\bbb{2}{3}{4},
\end{array}
\end{align*}
and $\bw_{\tau}(\Gamma) = -2\bbb{1}{2}{3} -2 \bbb{1}{2}{4}$. One can readily check that $ -2\bbb{1}{2}{3} -2 \bbb{1}{2}{4}$ is not in the $\Z$-span of the vectors listed above, and hence $L_3$ is not Ceresa-Zharkov trivial. 
\end{proof}

\begin{lemma}
\label{lem:K4L3minor}
Suppose $G$ is a graph with a minor $G'$ obtained by adding (parallel) edges to the complete graph $K_m$. Then there are edges $S_1\subset E(G)$ and $S_2 \subset E(G/S_1)$ such that $(G/S_1) \setminus S_2$ is $G'$, and every edge in $S_2$ is either a loop or parallel to another edge in $G/S_1$.  
\end{lemma}

\begin{proof}
The order in which edges are contracted or removed does not matter when forming a graph minor, so suppose $G' = (G/S_1) \setminus S_2$. As contracting an edge drops the number of vertices by 1 and removing an edge preserves the number of vertices, we have that $V(G/S_1) = V(G')$. As there is an edge between any two vertices in $G'$, every edge in $S_1$ must be either a loop or parallel to some other edge. 
\end{proof}

\begin{theorem}
\label{thm:WCTiffHET}
	A connected graph $G$ is Ceresa-Zharkov trivial if and only if it is of hyperelliptic type. 
\end{theorem}

\begin{proof}
Suppose $G$ is not of hyperelliptic type. By \cite[Theorem 1.1]{Corey}, $G$ has a $K_4$ or $L_3$ minor. By Lemma \ref{lem:K4L3minor}, there are subsets $S_1\subset E(G)$ and $S_2\subset E(G/S_1)$ such that $(G/S_1) \setminus S_2$ is $K_4$ or $L_3$, and every edge of $S_2$ is a loop or parallel to another edge in $G/S_1$. The graph $(G/S_1) \setminus S_2$ is not Ceresa-Zharkov trivial by Propositions \ref{prop:K4Nontrivial} and \ref{prop:L3Nontrivial}. So $(G/S_1)$ is not Ceresa-Zharkov trivial by Propositions \ref{prop:removeParallelEdge} and \ref{prop:WCT2connectedComponents}, and therefore $G$ is not Ceresa-Zharkov trivial by Proposition \ref{prop:weakCeresaContract}. 

For the converse, first suppose that $G$ is of hyperelliptic type and 2-connected.  By \cite[Theorem 4.5]{Corey}, there is a $\tilde{G}$ that is strongly of hyperelliptic type such that $G = \tilde{G}/S$ for some subset $S\subset E(G)$. The graph $\tilde{G}$ is Ceresa-Zharkov trivial by Proposition \ref{prop:tauSwapSeparatingPair}, and therefore so is $G$ by Proposition \ref{prop:weakCeresaContract}.  

In general, if $G$ is of hyperelliptic type, then its 2-connected components are of hyperelliptic type, and hence Ceresa-Zharkov trivial. The graph $G$ is Ceresa-Zharkov trivial by Proposition \ref{prop:WCT2connectedComponents}. 
\end{proof}

\begin{corollary}
\label{cor:WCTminorclosed}
	The property of being Ceresa-Zharkov trivial is a minor closed condition on graphs.
\end{corollary}

\begin{proof}
This follows from Theorem \ref{thm:WCTiffHET} and the fact that being of hyperelliptic type is a minor closed condition on graphs \cite[Proposition~3.8]{Corey}.
\end{proof}

\bibliographystyle{amsplain}
\bibliography{CELbib}

\end{document}